\documentclass[final,onefignum,onetabnum]{siamart220329}

\usepackage{amsmath,amsfonts,amsopn,amssymb}
\usepackage{graphicx}
\graphicspath{ {./Figures/} }
\usepackage[caption=false]{subfig}
\usepackage{multirow,relsize,makecell}
\usepackage{url}
\ifpdf
  \DeclareGraphicsExtensions{.eps,.pdf,.png,.jpg}
\else
  \DeclareGraphicsExtensions{.eps}
\fi
\usepackage{algorithm, algorithmic}
\usepackage{pgfplots}
\pgfplotsset{width=0.45\hsize,compat=1.18}

\numberwithin{theorem}{section}
\newsiamremark{remark}{Remark}
\newsiamremark{assumption}{Assumption}
\newsiamremark{example}{Example}
\crefname{remark}{Remark}{Remarks}
\crefname{assumption}{Assumption}{Assumptions}
\crefname{example}{Example}{Examples}

\headers{Two-level Schwarz for $2\MakeLowercase{m}$th-order elliptic}{Jongho Park}

\title{Two-level overlapping Schwarz preconditioners with universal coarse spaces for $2\MakeLowercase{m}$th-order elliptic problems\thanks{Submitted to arXiv.
}}

\author{
Jongho Park\thanks{Applied Mathematics and Computational Sciences Program, Computer, Electrical and Mathematical Science and Engineering Division, King Abdullah University of Science and Technology~(KAUST), Thuwal 23955, Saudi Arabia
 (\email{jongho.park@kaust.edu.sa}, \url{https://sites.google.com/view/jonghopark}).}
 }

\ifpdf
\hypersetup{
  pdftitle={Two-level overlapping Schwarz preconditioners for 2mth-order elliptic problems with universal coarse spaces},
  pdfauthor={Jongho Park}
}
\fi

\newcommand\rT{\mathrm{T}}

\newcommand\cI{\mathcal{I}}
\newcommand\cT{\mathcal{T}}
\newcommand\tS{\widetilde{S}}

\newcommand\sumT{\sum_{T \in \mathcal{T}_h}}

\usepackage[normalem]{ulem}
\usepackage{color}

\begin{document}

\maketitle

\begin{abstract}
We propose a novel universal construction of two-level overlapping Schwarz preconditioners for $2m$th-order elliptic boundary value problems, where $m$ is a positive integer.
The word ``universal" here signifies that the coarse space construction can be applied to any finite element discretization for any $m$ that satisfies some common assumptions.
We present numerical results for conforming, nonconforming, and discontinuous Galerkin-type finite element discretizations for high-order problems to demonstrate the scalability of the proposed two-level overlapping Schwarz preconditioners.
\end{abstract}



\section{Introduction}
\label{Sec:Introduction}
Elliptic problems of order $2m$, where $m \in \mathbb{Z}_{>0}$, which are natural generalizations of the well-known Poisson equation, have theoretical importance and find applications in a broad range of science and engineering.
For instance, fourth-order elliptic problems~($m = 2$) arise in continuum mechanics, including the theory of linear elasticity~\cite{Brenner:1996,Ciarlet:2002}, and in optimal control theory~\cite{BS:2017,LGY:2009}.
On the other hand, sixth-order elliptic problems~($m = 3$) are encountered in mathematical models within material sciences, such as in the oxidation of silicon~\cite{BLN:2004,King:1989} and in the phase-field crystal model~\cite{CW:2008,WWL:2009}.

Due to the importance of $2m$th-order elliptic problems, there has been a vast literature on numerical methods for them.
For fourth-order problems, i.e., when $m = 2$, various finite element methods including conforming~\cite{AFS:1968,BFS:1965,Valdman:2020}, nonconforming~\cite{HS:2013,LL:1975,WSX:2007}, and discontinuous Galerkin-type~\cite{BS:2005,EGHLMT:2002} have been studied.
There have also been a number of existing works on finite element methods for sixth-order problems~($m = 3$), including nonconforming methods~\cite{JW:2023}, interior penalty methods~\cite{GN:2011}, and mixed methods~\cite{DILW:2019}.
Recently, universal approaches that can be applied to any $2m$th-order problems have been developed, such as conforming finite element discretizations on tensor product elements~\cite{HZ:2015} and simplical elements~\cite{HLW:2023}, nonconforming elements~\cite{WX:2013,WX:2019}, virtual elements~\cite{CH:2020}, and discontinuous Galerkin methods~\cite{CLQ:2022}.
Moreover, neural network-based approaches have also been considered; see~\cite{HX:2023,Xu:2020}.

As the mesh size decreases in finite element methods for $2m$th-order problems, the condition number typically increases dramatically.
This underscores the importance of designing effective preconditioners for solving large-scale problems with a high number of degrees of freedom.
In this viewpoint, this paper focuses on developing two-level overlapping Schwarz preconditioners for various finite element discretizations of $2m$th-order elliptic boundary value problems.
The Schwarz method has proven to be a powerful parallel algorithm for solving large-scale problems encountered in partial differential equations and related scientific problems over the past decades.
Existing results on second-order problems~($m=1$) can be found in, e.g.,~\cite{LP:2023,TW:2005} and references therein.
For fourth-order problems~($m=2$), additive Schwarz preconditioners for conforming, nonconforming, and discontinuous Galerkin-type discretizations were proposed in \cite{Brenner:1996,BW:2005,Zhang:1996}.
Meanwhile, several works have extended the application of Schwarz methods to nonlinear problems; for optimization-based approaches, see~\cite{LP:2022,Park:2023,Park:2024,TX:2002}, and for nonlinear preconditioning-based approaches, refer to~\cite{CK:2002,CL:2011,DGKKM:2016,HC:2007}.
However, to the best of knowledge, there have been no existing works on Schwarz preconditioners that can be applied to the case of general $m$ and its discretizations.

In this paper, we introduce novel two-level overlapping Schwarz preconditioners for various finite element discretizations of $2m$th-order elliptic problems.
A key aspect of our approach lies in the construction of coarse spaces.
These coarse spaces are built using conforming finite element spaces~\cite{HLW:2023,HZ:2015} defined on a coarse mesh, with a construction method that remains independent of both $m$ and the type of finite element discretization employed.
Thus, the construction of coarse spaces in our approach can be deemed ``universal'', as it is applicable to any finite element discretization for any $m$.
Equipped with the universal coarse spaces, we demonstrate that the proposed preconditioners ensure uniformly bounded condition numbers of the preconditioned operators when $H/\delta$ is fixed, where $H$ represents the typical diameter of a subdomain and $\delta$ measures the overlap among the subdomains.
To validate the numerical performance of the proposed preconditioners across various finite elements, we present numerical results for fourth- and sixth-order problems discretized using the Bogner--Fox--Schmit~(BFS)~\cite{BFS:1965,Valdman:2020,WSX:2007}, Adini~\cite{HS:2013,LL:1975}, $C^0$ interior penalty~\cite{BS:2005,EGHLMT:2002}, and Jin--Wu~\cite{JW:2023} finite elements.

The rest of this paper is organized as follows.
In \cref{Sec:Model}, we introduce a model $2m$th-order elliptic problem and its abstract finite element discretization.
In \cref{Sec:Schwarz}, we provide a brief summary of the abstract theory of overlapping Schwarz preconditioners.
A domain decomposition setting that is used in the construction of the proposed two-level overlapping Schwarz preconditioners is presented in \cref{Sec:2L}.
Convergence analysis for the proposed two-level overlapping Schwarz preconditioners is presented in \cref{Sec:Convergence}.
A detailed description on the universal construction of coarse spaces is given in \cref{Sec:Universal}.
Applications of the proposed preconditioners to high-order conforming, nonconforming, and discontinuous Galerkin-type elements are presented in \cref{Sec:Applications}.
Finally, we conclude the paper with remarks in \cref{Sec:Conclusion}.

\section{\texorpdfstring{$2m$}{2m}th-order elliptic problems}
\label{Sec:Model}
In this section, we introduce a model $2m$th-order elliptic boundary value problem and its finite element discretizations.

In what follows, we use the notation $A \lesssim B$ and $B \gtrsim A$ to represent that there exists a constant $c > 0$ independent of the geometric parameters $h$, $H$, and $\delta$, which will be precisely defined later, such that $A \leq c B$.
We write $A \approx B$ if and only if $A \lesssim B$ and $A \gtrsim B$.

Let $\Omega \subset \mathbb{R}^d$~($d \in \mathbb{Z}_{>0})$ be a bounded polyhedral domain.
We consider the following polyharmonic equation with a homogeneous boundary condition:
\begin{equation}
\label{model_strong}
\begin{split}
(- \Delta )^m u = f & \quad \text{ in } \Omega, \\
u = \frac{\partial u}{\partial \nu} = \dots = \frac{\partial^{m-1} u }{ \partial \nu^{m-1}} = 0 & \quad \text{ on } \partial \Omega,
\end{split}
\end{equation}
where $f \in L^2 (\Omega)$ and $\nu$ is the outward unit normal vector field along $\partial \Omega$.
The problem~\eqref{model_strong} admits the following weak formulation: find $u \in H_0^m (\Omega)$ such that
\begin{equation}
\label{model_cont}
\int_{\Omega} \nabla^m u : \nabla^m v \,dx = \int_{\Omega} fv \,dx
\quad \forall v \in H_0^m (\Omega).
\end{equation}
Well-posedness of the problem~\eqref{model_cont} is ensured by the Lax--Milgram theorem~\cite{CH:2020}.

\subsection{Finite element discretizations}
In order to solve~\eqref{model_cont} numerically, we need to consider a suitable discretization of~\eqref{model_cont}.
Here, extending the result for fourth-order problems presented in~\cite{Park:2023}~(see also~\cite{BS:2017,BSZ:2012}), we present an abstract finite element discretization for~\eqref{model_cont} that is suitable for scalable two-level overlapping Schwarz preconditioners.
Examples of existing finite elements that adhere to the abstract framework will be provided in \cref{Sec:Applications}.

We consider the following abstract finite element discretization of~\eqref{model_cont}: find $u_h \in S_h$ such that
\begin{equation}
\label{model_FEM}
a_h (u_h, v) = \int_{\Omega} fv \,dx
\quad \forall v \in S_h,
\end{equation}
where $S_h \subset L^2 (\Omega)$ is a finite element space defined on a quasi-uniform triangulation\footnotemark[1]\footnotetext[1]{In some finite elements for~\eqref{model_cont}, the reference element is not a simplex but a polytope~(see \cref{Sec:Applications}).
Nevertheless, for the sake of convenience, we use the terminology ``triangulation" even in these cases.} $\cT_h$ of the domain $\Omega$, and $h$ stands for the characteristic element diameter.
In the following, we summarize some key assumptions on~\eqref{model_FEM}, which are essential for the convergence analysis of two-level overlapping Schwarz preconditioners~(cf.~\cite[Assumption~2.1]{Park:2023}).

\begin{assumption}

\label{Ass:FEM}
In the finite element discretization~\eqref{model_FEM}, we have the following:
\begin{itemize}
\item There exists a norm $\| \cdot \|_h$ defined on $S_h + H_0^m (\Omega)$ that satisfies
\begin{equation}
\label{norm}
a_h (u, u) \approx \| u \|_h^2, \quad u \in S_h,
\end{equation}
which implies that $a_h  (\cdot, \cdot)$ is equivalent to $ \| \cdot \|_h^2$ in $S_h$.

\item There exist a conforming finite element space $\tS_h \subset H_0^m (\Omega)$ and an enriching operator $E_h \colon S_h \rightarrow \tS_h$ such that
\begin{multline}
    \label{E_h}
    \| u - E_h u \|_{L^2 (\Omega)} + \sum_{j=1}^{m-1} h^j \left( \sumT |u - E_h u|_{H^j (T)}^2 \right)^{\frac{1}{2}} \\
    + h^m |E_h u |_{H^m (\Omega)}
    \lesssim h^m \| u \|_h, \text{ } u \in S_h.
\end{multline}
\end{itemize}
\end{assumption}

Indeed, the above assumptions are commonly satisfied in many existing finite element methods for~\eqref{model_cont}.
One may refer to~\cite{BS:2017,BSZ:2012} for case studies on nonconforming and discontinuous Galerkin methods for fourth-order problems.
Note that~\eqref{norm} implies that the bilinear form $a_h (\cdot, \cdot)$ is coercive in $S_h$, ensuring that~\eqref{model_FEM} has a unique solution $u_h \in S_h$. 

\begin{remark}
\label{Rem:FEM}
To guarantee convergence to a continuous solution, we need additional assumptions on~\eqref{model_FEM}, such as elliptic regularity, interpolation estimates, and the approximability of $a_h (\cdot, \cdot)$ for $a( \cdot, \cdot)$~\cite{BS:2017,BSZ:2012}.
However, as these assumptions are not required for the convergence analysis of overlapping Schwarz preconditioners, we omit them here.
\end{remark}

\begin{remark}
\label{Rem:pointwise}
Compared to \cref{Ass:FEM}, in~\cite[Assumption~2.1]{Park:2023}, there is an additional requirement that the enriching operator $E_h$ should preserve the function values at the vertices of $\cT_h$.
This condition becomes necessary when dealing with constrained problems such as variational inequalities. 
However, in this paper, such assumptions are not needed because we are addressing the linear model problem~\eqref{model_cont}.
\end{remark}

\section{Overlapping Schwarz preconditioners}
\label{Sec:Schwarz}
In this section, we present a brief overview of the abstract framework of two-level overlapping Schwarz preconditioners introduced in~\cite{TW:2005}.
One may refer to~\cite{BS:2008,Xu:1992} for alternative representations.
Throughout this section, we employ a slight abuse of notation by not distinguishing between finite element functions and the corresponding vectors of degrees of freedom. 
The same applies for discrete operators and their matrix representations.

Let $A$ and $f$ be the stiffness matrix and the load vector induced by $a_h (\cdot, \cdot)$ and $\int_{\Omega} f \cdot \,dx$ in~\eqref{model_FEM}, respectively, i.e.,
\begin{equation*}
    u^T A v = a_h (u,v), \quad
    f^T u = \int_{\Omega} f u \,dx,
\end{equation*}
for all $u, v \in S_h$.
Then the variational formulation~\eqref{model_FEM} is equivalent to the linear system
\begin{equation}
\label{model_linear}
    A u = f.
\end{equation}

Next, we assume that the solution space $V = S_h$ of~\eqref{model_linear} admits a decomposition of the form
\begin{equation*}
    V = \sum_{k=0}^N R_k^T V_k,
\end{equation*}
where each $V_k$, $0 \leq k \leq N$, is a finite-dimensional space and $R_k^T \colon V_k \rightarrow V$ is a suitable prolongation operator.
The space $V_0$ is referred to as the coarse space and is constructed on a coarse mesh.
On the other hand, the remaining spaces are referred to as local spaces and are associated with local problems defined on subdomains.

In this setting, the two-level additive Schwarz preconditioner $M^{-1}$ is defined as
\begin{equation}
\label{preconditioner}
M^{-1} = \sum_{k=0}^N R_k^{\rT} A_k^{-1} R_k,
\end{equation}
where $A_k = R_k A R_k^{\rT}$, $0 \leq k \leq N$.
It is well-known that the condition number $\kappa (M^{-1} A)$ of the preconditioned matrix $M^{-1} A$ can be estimated by a stable decomposition argument; see~\cite[Theorem~2.7]{TW:2005} and~\cite[Section~4.1]{Park:2020}.

\begin{lemma}
\label{Lem:ASM}
Suppose that there exists a constant $C_0$, such that any $u \in V$ admits a stable decomposition
\begin{equation*}
u = \sum_{k=0}^N R_k^{\rT} u_k, \quad u_k \in V_k, \text{ } 0 \leq k \leq N,
\end{equation*}
that satisfies
\begin{equation*}
\sum_{k=0}^N a_h ( R_k^{\rT} u_k, R_k^{\rT} u_k ) \leq C_0^2 a_h (u, u).
\end{equation*}
In addition, suppose that there exists a constant $N_c$ that satisfies
\begin{equation*}
    a_h \left( \sum_{k=1}^N R_k^T u_k, \sum_{k=1}^N R_k^T u_k \right) \leq N_c \sum_{k=1}^N a_h (R_k^T u_k, R_k^T u_k),
    \quad u_k \in V_k, \text{ } 1 \leq k \leq N.
\end{equation*}
Then we have
\begin{equation*}
\kappa (M^{-1} A) \leq C_0^2 (N_c + 1).
\end{equation*}
\end{lemma}

Typically, in overlapping domain decomposition methods, the constant $N_c$ in \cref{Lem:ASM} can be estimated by a coloring argument~\cite{Park:2020,TW:2005}, and is independent of $N$ or other geometric parameters.
Hence, to analyze the convergence of the additive Schwarz preconditioner, it suffices to estimate the stable decomposition parameter $C_0$ in \cref{Lem:ASM}.

\begin{remark}
\label{MSM}
The multiplicative Schwarz preconditioner, where local and coarse problems are solved sequentially, can be analyzed in terms of the constants $C_0$ and $N_c$ appearing in \cref{Lem:ASM} as well; see \cite{BPWX:1991} and \cite[Theorem 2.9]{TW:2005}.
Therefore, in this paper, we focus solely on the additive Schwarz preconditioner.
\end{remark}

\section{Two-level domain decomposition}
\label{Sec:2L}
In this section, we present a domain decomposition setting for the two-level additive Schwarz preconditioner~\eqref{preconditioner}.
Let $\cT_H$ be a quasi-uniform triangulation of $\Omega$ such that $\cT_h$ is a refinement of $\cT_H$, where $H$ stands for the characteristic element diameter of $\cT_H$.
The domain $\Omega$ is decomposed into $N$ overlapping subdomains $\{ \Omega_k \}_{k=1}^N$, where each subdomain $\Omega_k$, $1 \leq k \leq N$, is a union of $\cT_h$-elements.
We assume that $\operatorname{diam} \Omega_k \approx H$.
The overlap among the subdomains is measured by a parameter $\delta > 0$.

\subsection{Local spaces}
We define the local spaces $\{ V_k \}_{k=1}^N$ used in the two-level additive Schwarz preconditioner~\eqref{preconditioner} as follows.
Let $V = S_h$, and we set
\begin{equation}
\label{V_k}
V_k = S_h (\Omega_k), \quad 1 \leq k \leq N,
\end{equation}
where $S_h (\Omega_k) $ is the finite element space of the same type as $S_h$, but defined on the restriction of $\cT_h$ in $\Omega_k$.
We also set the prolongation operator $R_k^T$ as the natural extension operator from $S_h (\Omega_k)$ to $S_h$. 
We present below essential assumptions on the local spaces $\{ V_k \}_{k=1}^N$ for two-level overlapping Schwarz preconditioners~(cf.~\cite[Assumption~3.1]{Park:2023}).

\begin{assumption}
\label{Ass:local}
In the local spaces~\eqref{V_k}, we have the following:
\begin{itemize}
\item The local spaces $\{ V_k \}_{k=1}^N$ can be colored with a number of colors $N_c$ independently of $N$.

\item For any $u \in S_h$, there exists a decomposition $u = \sum_{k=1}^N R_k^T u_k$, $u_k \in R_k^T V_k$, such that
\begin{multline*}
\sum_{k=1}^N \| R_k^T u_k \|_h^2 \lesssim \| u \|_h^2 + \sum_{j=1}^{m-1} \frac{1}{\delta^{2(m-j)}} \left( \sumT |u|_{H^j (T)}^2 \right)  \\
+ \frac{1}{H\delta^{2m-1}} \| E_h u \|_{L^2 (\Omega)}^2 + \left( \frac{H}{\delta} \right)^{2m-1} |E_h u|_{H^m (\Omega)}^2.
\end{multline*}
\end{itemize}
\end{assumption}

The coloring condition in \cref{Ass:local} can be found in, e.g.,~\cite[Section~2.5.1]{TW:2005}.
On the other hand, usually, the decomposition condition in \cref{Ass:local} is derived by using a $W^{m,\infty}$-partition of unity subordinate to the domain decomposition $\{ \Omega_k \}_{k=1}^N$ and invoking a trace theorem-type argument developed in~\cite{Brenner:1996,DW:1994,TW:2005}.
One may refer to, for instance, \cite{TW:2005} and \cite{Brenner:1996,BW:2005} for derivations of the decomposition conditions for second- and fourth-order problems, respectively.
In the following, we provide a generalization of the trace theorem-type argument for $H^m$-functions, which provides an estimate of the $L^2$-norm of a function over a strip along the boundary of the domain.

\begin{lemma}
\label{Lem:trace}
Let $D \subset \mathbb{R}^d$ ($d \in \mathbb{Z}_{>0}$) be a bounded polyhedral domain with diameter  $H$, and let $D_{\delta} \subset D$ be the set of points that are within a distance $\delta$ of $\partial D$.
Then we have
\begin{equation*}
    \| u \|_{L^2 (D_{\delta})}^2 \lesssim \frac{\delta}{H} \left( \| u \|_{L^2 (D)}^2 + H^{2m} |u|_{H^m (D)}^2 \right),
    \quad u \in H^m (D).
\end{equation*}
\end{lemma}
\begin{proof}
This proof closely follows the argument in~\cite[Lemma~3.10]{TW:2005}.
We may assume that $\delta \leq H$.
Take any $u \in H^m (D)$.
We cover $D_{\delta}$ by shape-regular patches with $\mathcal{O} (\delta)$ diameters.
By invoking the Friedrichs inequality presented in \cref{Thm:Friedrichs} for each patch and summing over the patches, we obtain
\begin{equation}
\label{Lem1:trace}
    \| u \|_{L^2 (D_{\delta})}^2 \lesssim \delta^{2m} | u |_{H^m (D_{\delta})}^2 + \sum_{j=0}^{m-1} \delta^{2j+1} \left\| \frac{\partial^j u}{\partial \nu^j} \right\|_{L^2 (\partial D)}^2,
\end{equation}
where $\nu$ is the outward unit normal vector field along $\partial D$.
By the trace theorem, for each $j$, we have
\begin{equation}
\label{Lem2:trace}
\left\| \frac{\partial^j u}{\partial \nu^j} \right\|_{L^2 (\partial D)}^2
\lesssim H | u |_{H^{j+1} (D)}^2 + \frac{1}{H} | u |_{H^j (D)}^2.
\end{equation}
Combining~\eqref{Lem1:trace} and~\eqref{Lem2:trace} yields
\begin{equation*}
\begin{split}
\| u \|_{L^2 (D_{\delta})}^2
&\lesssim \delta^{2m} | u |_{H^m (D_{\delta})}^2 + \sum_{j=0}^{m-1} \delta^{2j+1} \left( H | u |_{H^{j+1} (D)}^2 + \frac{1}{H} | u |_{H^j (D)}^2 \right) \\
&\lesssim \frac{\delta}{H} \left( \| u \|_{L^2 (D)}^2 + \sum_{j=1}^{m} H^{2j} | u |_{H^j (D)}^2 \right) \\
&\lesssim \frac{\delta}{H} \left( \| u \|_{L^2 (D)}^2 + H^{2m} | u |_{H^m (D)}^2 \right),
\end{split}
\end{equation*}
where the last inequality is due to~\cite[Theorem~1.8]{Necas:2012}.
This completes the proof.
\end{proof}

\subsection{Coarse space}
In two-level overlapping Schwarz preconditioners, the coarse space $V_0$ is typically defined by a finite element space on the coarse mesh $\cT_H$.
A common choice for $V_0$ is $S_H$, which is the finite element space of the same type as $S_h$ but defined on $\cT_H$.
However, in this section, we introduce an alternative choice for $V_0$ that is compatible with any finite element discretizations fitting into the framework introduced in \cref{Sec:Model}.

Let $\{ x^i \}_{i \in \cI_H}$ denote the collection of all vertices of $\cT_H$, where $\cI_H$ represents the set of indices.
For each $i \in \cI_H$, a region $\omega_i \subset \Omega$ is defined as the union of the coarse elements $T \in \cT_H$ such that $x^i \in \partial T$:
\begin{equation}
    \label{omega_i}
    \overline{\omega}_i = \bigcup_{T \in \cT_H, \text{ } x^i \in \partial T} \overline{T}.
\end{equation}
Utilizing existing results on the construction of conforming finite elements in any dimension~\cite{HLW:2023,HZ:2015}, we are able to explicitly find a collection $\{ \phi_i \}_{i \in \cI_H}$ of functions in $W^{m, \infty} (\Omega)$ such that
\begin{subequations}
\label{coarse_basis}
\begin{align}
\phi_i = 0& \quad \text{ on } \Omega \setminus \omega_i, \label{coarse_basis1} \\
\sum_{i \in \cI_H} \phi_i = 1& \quad \text{ on } \overline{\Omega}, \label{coarse_basis2} \\
| \phi_i |_{W^{j, \infty} (\omega_i)} \lesssim \frac{1}{H^j},&
\quad 1 \leq j \leq m, \text{ } i \in \cI_H.\label{coarse_basis3}
\end{align}
\end{subequations}
We will present how to construct $\{ \phi_i \}_{i \in \cI_H}$ using conforming finite element basis functions in \cref{Sec:Universal}.

Now, we define the coarse space $V_0$ as follows:
\begin{equation}
    \label{V_0}
    V_0 = \left( \sum_{i \in \cI_H} \phi_i \mathbb{P}_{m-1} (\omega_i) \right) \cap H_0^m (\Omega),
\end{equation}
where $\mathbb{P}_{m-1} (\omega_i)$ represents the collection of $(m-1)$th-degree polynomials defined on $\omega_i$.
We note that a similar type of coarse spaces was considered in~\cite{Park:2023}.
As $V_0 \subset H_0^m (\Omega)$, it can be regarded as a conforming finite element space on $\cT_H$.
We also mention that the concept of employing conforming coarse spaces for nonconforming methods was previously examined in~\cite{Lee:1993}.

\begin{remark}
\label{Rem:pou}
Different from~\cite{Park:2023}, the collection $\{ \phi_i \}_{i \in \cI_H}$ need not to be nonnegative.
That is, each $\phi_i$ may have negative function values.
This difference makes the construction of $\{ \phi_i \}_{i \in \cI_H}$ more tractable, while the construction of a nonnegative smooth partition of unity on a general mesh with an explicit formulation is a rather challenging task. 
In~\cite{KO:2017}, a nonnegative smooth partition of unity with an explicit formulation for the cartesian grid was introduced.
\end{remark}

\begin{remark}
\label{Rem:GDSW}
The proposed coarse space~\eqref{V_0} has some resemblance to the generalized generalized Dryja--Smith--Widlund~(GDSW) coarse space~\cite{DKW:2008a,HKKR:2019} in that both are defined in terms of partition of unity and can be used with various finite element discretizations.
The GDSW coarse space is composed of the discrete harmonic extensions of partition of unity functions on the interface of a nonoverlapping domain decomposition.
This makes it discretization-dependent and it varies with different finite element discretizations.
In contrast, the proposed coarse space~\eqref{V_0} is common for all finite element discretizations.
Additionally, while the GDSW coarse space does not require a coarse mesh but rather a nonoverlapping domain decomposition, making it more suitable for less regular subdomains~\cite{DKW:2008b}, the proposed coarse space relies on a conforming finite element space defined on the coarse mesh $\cT_H$.
\end{remark}

The following assumption summarizes the existence of a coarse prolongation operator $R_0^T$, which maps the coarse space $V_0$ to the fine space $V$, with certain approximation properties.
Note that $R_0^T V_0 \subset V$, although $V_0$ may not be a subspace of $V$.
It is worth mentioning that this assumption was also considered in~\cite[Assumption~4.1]{Park:2023}.
As we will discuss in \cref{Sec:Applications}, in most applications, selecting $R_0^T$ as the nodal interpolation operator to $V$ ensures that \cref{Ass:coarse} holds.

\begin{assumption}
\label{Ass:coarse}
The coarse prolongation operator $R_0^T \colon V_0 \rightarrow V$ satisfies
\begin{multline}
\label{R_0^T}
\| u - R_0^T u \|_{L^2 (\Omega)} + \sum_{j=1}^{m-1} H^j \left( \sumT | u - R_0^T u |_{H^j (T)}^2 \right)^{\frac{1}{2}} + H^m \| R_0^T u \|_h \\
\lesssim H^m | u |_{H^m (\Omega)}, \quad u \in V_0.
\end{multline}
\end{assumption}

\section{Convergence analysis}
\label{Sec:Convergence}
In this section, we provide an estimate for the condition number of the preconditioned matrix $M^{-1} A$, where $M^{-1}$ is the two-level additive Schwarz preconditioner given in~\eqref{preconditioner}.

In the analysis of two-level overlapping Schwarz preconditioners, as in many existing works~\cite{Brenner:1996,BW:2005,TW:2005,Zhang:1996}, a crucial aspect is to examine the approximability and stability of an interpolation operator onto the coarse space.
In \cref{Lem:J_H}, we establish that the coarse space $V_0$ defined in~\eqref{V_0} admits a quasi-interpolation operator with favorable approximability and stability estimates.

\begin{lemma}
\label{Lem:J_H}
Let $V_0$ be the coarse space defined in~\eqref{V_0}.
There exists a quasi-interpolation operator $J_H \colon H_0^m (\Omega) \rightarrow V_0$ that satisfies
\begin{equation}
\label{Lem1:J_H}
    \| u - J_H u \|_{L^2 (\Omega)} + \sum_{j=1}^{m-1} H^j | u - J_H u |_{H^j (\Omega)} + H^m |J_H u|_{H^m (\Omega)}
    \lesssim H^m | u |_{H^m (\Omega)},
    \text{ } u \in H_0^m (\Omega).
\end{equation}
\end{lemma}
\begin{proof}
Take any $u \in H_0^m (\Omega)$ and $i \in \mathcal{I}_H$.
We begin by constructing a local polynomial approximation $J_i u \in \mathbb{P}_{m-1} (\overline{\omega}_i)$ of $u$ on $\omega_i$, where the definition of $\omega_i$ was given in~\eqref{omega_i}, that satisfies the following:
\begin{equation}
\label{Lem2:J_H}
    \| u - J_i u \|_{L^2 (\omega_i)} + \sum_{j=1}^{m-1} H^j | u - J_i u |_{H^j (\omega_i)} + H^m |u|_{H^m (\omega_i)}
    \lesssim H^m | u |_{H^m (\omega_i)}.
\end{equation}
If $\partial \omega_i \cap \partial \Omega$ has nonzero measure, then setting $J_i u = 0$ satisfies~\eqref{Lem2:J_H}, thanks to the Friedrichs inequality presented in \cref{Thm:Friedrichs}.
Otherwise, if $\partial \omega_i \cap \partial \Omega$ has zero measure, we can find an $(m-1)$th-degree polynomial $J_i u$ satisfying~\eqref{Lem2:J_H} by invoking the Bramble--Hilbert lemma~\cite[Lemma~4.3.8]{BS:2008}.

We define $J_H u \in V_0$ as
\begin{equation*}
    J_H u = \sum_{i \in \mathcal{I}_H} (J_i u) \phi_i,
\end{equation*}
where $\phi_i$ was given in~\eqref{coarse_basis}.
Due to the construction of $J_i u$, it is clear that $J_H u$ satisfies the homogeneous boundary condition on $\partial \Omega$.
In the following, we employ a similar argument as in~\cite[Lemma~4.5]{Park:2023} to derive~\eqref{Lem1:J_H}.

For a coarse element $T \in \cT_H$, let $\{ x^i \}_{i=1}^{n_T}$ be the vertices of $T$, where $n_T$ is uniformly bounded with respect to $T$ because $\cT_H$ is quasi-uniform.
If we define a region $\omega_T \subset \Omega$ as
\begin{equation*}
    \overline{\omega}_T = \bigcup_{i=1}^{n_T} \overline{\omega}_i,
\end{equation*}
then for $1 \leq i \leq n_T$ and $0 \leq j \leq m$, we have
\begin{equation}
\label{Lem3:J_H}
| ( u - J_i u ) \phi_i |_{H^j (\omega_i)}
\lesssim \sum_{l=0}^j | u - J_i u |_{H^l (\omega_i)} | \phi_i |_{W^{j-l, \infty} (\omega_i)}
\lesssim H^{m-j} | u |_{H^m (\omega_i)},
\end{equation}
where the second inequality is because of~\eqref{Lem2:J_H} and~\eqref{coarse_basis3}.
It follows that
\begin{equation}
\label{Lem4:J_H}
| u - J_H u |_{H^j (T)} 
\stackrel{\eqref{coarse_basis2}}{\leq} \sum_{i=1}^{n_T} | (u - J_i u ) \phi_i |_{H^j (\omega_i)}
\stackrel{\eqref{Lem3:J_H}}{\lesssim} H^{m-j} | u |_{H^m (\omega_T)}.
\end{equation}
Summing~\eqref{Lem4:J_H} over all $T$ yields
\begin{equation}
\label{Lem5:J_H}
    | u - J_H u |_{H^j (\Omega)} \lesssim H^{m-j} | u |_{H^m (\Omega)}.
\end{equation}
Finally, by combining~\eqref{Lem5:J_H} and
\begin{equation*}
    | J_H u |_{H^m (\Omega)} \leq | u - J_H u |_{H^m (\Omega)} + | u |_{H^m (\Omega)}
    \stackrel{\eqref{Lem5:J_H}}{\lesssim} | u |_{H^m (\Omega)},
\end{equation*}
we complete the proof of~\eqref{Lem1:J_H}.
\end{proof}

Now, we present a condition number estimate for the preconditioned matrix $M^{-1} A$ in \cref{Thm:kappa}.
The proof of \cref{Thm:kappa} can be done by utilizing the approximation properties of $J_H$ presented in \cref{Lem:J_H}.

\begin{theorem}
\label{Thm:kappa}
Suppose that the following conditions hold:
\begin{itemize}
\item The finite element discretization~\eqref{model_FEM} satisfies \cref{Ass:FEM}.
\item The local spaces~\eqref{V_k} satisfy \cref{Ass:local}.
\item The coarse space~\eqref{V_0} satisfies \cref{Ass:coarse}.
\end{itemize}
Then we have
\begin{equation*}
\kappa (M^{-1} A) \lesssim \left( 1 + \left( \frac{H}{\delta} \right)^{2m-1} \right),
\end{equation*}
where $M^{-1}$ was given in~\eqref{preconditioner}.
\end{theorem}
\begin{proof}
Thanks to \cref{Lem:ASM} and \cref{Ass:local}, it suffices to find a stable decomposition estimate for $u \in V$.
We set $u_0 = J_H (E_h u)$, where $J_H$ is the quasi-interpolation operator given in \cref{Lem:J_H}.
We observe that
\begin{equation}
    \label{Thm1:kappa}
    | J_H (E_h u) |_{H^m (\Omega)}^2
    \stackrel{\eqref{Lem1:J_H}}{\lesssim} | E_h u |_{H^m (\Omega)}^2
    \stackrel{\eqref{E_h}}{\lesssim} \| u \|_h^2.
\end{equation}
Then we can estimate $a_h (R_0^T u_0, R_0^T u_0)$ as follows:
\begin{equation}
    \label{Thm2:kappa}
    a_h (R_0^T u_0, R_0^T u_0)
    \stackrel{\eqref{norm}}{\approx} \| R_0^T J_H (E_h u) \|_h^2
    \stackrel{\eqref{R_0^T}}{\lesssim} | J_H (E_h u) |_{H^m (\Omega)}^2
    \stackrel{\eqref{Thm1:kappa}}{\lesssim} \| u \|_h^2.
\end{equation}

Meanwhile, by \cref{Ass:local}, there exist $u_k \in V_k$, $1 \leq k \leq N$, that satisfy $u - R_0^T u_0 = \sum_{k=1}^N R_k^T u_k$ and 
\begin{equation} \begin{split}
    \label{Thm3:kappa}
    \sum_{k=1}^N a_h &(R_k^T u_k, R_k^T u_k)
    \stackrel{\eqref{norm}}{\approx} \sum_{k=1}^N \| R_k^T u_k \|_h^2 \\
    &\lesssim \| u - R_0^T u_0 \|_h^2 + \sum_{j=1}^{m-1} \left( \frac{1}{\delta^{2(m-j)}} \sumT | u - R_0^T u_0 |_{H^j (T)}^2 \right) \\
    &\quad + \frac{1}{H \delta^{2m-1}} \| E_h (u - R_0^T u_0) \|_{L^2 (\Omega)}^2 + \left( \frac{H}{\delta} \right)^{2m-1} | E_h (u - R_0^T u_0) |_{H^m (\Omega)}^2.
\end{split} \end{equation}
In order to estimate the rightmost-hand side of~\eqref{Thm3:kappa}, we estimate some norms of $u - R_0^T u_0$; using~\eqref{E_h},~\eqref{Lem1:J_H},~\eqref{R_0^T}, and~\eqref{Thm1:kappa}, we get
\begin{subequations}
\label{Thm4:kappa}
\begin{equation} \begin{split}
&\| u - R_0^T u_0 \|_{L^2 (\Omega)}^2 \\
&\lesssim \| u - E_h u \|_{L^2 (\Omega)}^2 + \| E_h u - J_H (E_h u) \|_{L^2 (\Omega)}^2 + \| J_H (E_h u) - R_0^T J_H (E_h u) \|_{L^2 (\Omega)}^2 \\
&\lesssim \left( h^{2m} + H^{2m} \right) \| u \|_h^2
\lesssim H^{2m} \| u \|_h^2.
\end{split} \end{equation}
In the same manner, for each $j$, we have
\begin{equation} \begin{split}
    &\sumT | u - R_0^T u_0 |_{H^j (T)}^2  \\
    &\lesssim \sumT  | u - E_h u |_{H^j (T)}^2 + | E_h u - J_H (E_h u) |_{H^j (\Omega)}^2 \\
    &\quad + \sumT | J_H (E_h u) - R_0^T J_H (E_h u) |_{H^j (T)}^2 \\
    &\lesssim ( h^{2(m-j)} + H^{2(m-j)} ) \| u \|_h^2
    \lesssim H^{2(m-j)} \| u \|_h^2
\end{split} \end{equation}
and
\begin{equation}
\| u - R_0^T u_0 \|_h^2
\lesssim \| u \|_h^2 + \| R_0^T u_0 \|_h^2
\lesssim \| u \|_h^2 + | J_H (E_h u) |_{H^m (\Omega)}^2
\lesssim \| u \|_h^2.
\end{equation}
\end{subequations}
Then, invoking~\eqref{Thm3:kappa} with~\eqref{E_h} and~\eqref{Thm4:kappa} yields
\begin{equation}
    \label{Thm5:kappa}
    \sum_{k=1}^N a_h (R_k^T u_k, R_k^T u_k ) \lesssim \left[ 1 + \sum_{j=1}^{m-1} \left( \frac{H}{\delta} \right)^{2(m-j)} + \left( \frac{H}{\delta} \right)^{2m-1} \right] \| w \|_h^2.
\end{equation}
Combining~\eqref{Thm2:kappa} and~\eqref{Thm5:kappa} yields the desired result.
\end{proof}

\Cref{Thm:kappa} implies that the convergence rate of an iterative method preconditioned by the two-level additive Schwarz preconditioner has a uniform bound when $H/\delta$ is fixed.
In both cases of small overlap $\delta \approx h$ and generous overlap $\delta \approx H$, we can deduce that the method is scalable in the sense that the convergence rate depends only on the subdomain size $H/h$ and is independent of the full-dimension problem size $h$.

\section{Universal construction of coarse spaces}
\label{Sec:Universal}
In this section, we discuss the construction of the collection $\{ \phi_i \}_{i \in \cI_H}$, which is the main ingredient in defining the coarse space $V_0$ as described in~\eqref{V_0}.
The construction introduced here is considered universal as it offers a unified approach that can be applied to any finite element discretization for any $m$.

Let $\widehat{S}_H \subset H^m (\Omega)$ be a conforming finite element space defined on the coarse triangulation $\cT_H$, without any essential boundary conditions.
Examples of such conforming finite element spaces for general $m$ can be found in~\cite{HLW:2023} for simplical meshes and~\cite{HZ:2015} for rectangular meshes.
Then the constant function $1$ on $\Omega$ can be represented using basis functions for $\widehat{S}_H$ as follows:
\begin{equation}
\label{one}
1 = \sum_{j \in \widehat{\mathcal{I}}_{H}} \widehat{\phi}_j,
\end{equation}
where $\{ \widehat{\phi}_j \}_{j \in \widehat{\cI}_H}$ denotes the collection of all nodal basis functions for $\widehat{S}_H$ corresponding to function evaluation at points.
The construction of $\{ \phi_i \}_{i \in \cI_H}$ can be done by distributing the right-hand side of~\eqref{one} to each $\omega_i$ as evenly as possible.
Namely, for each $i \in \cI_H$, we set
\begin{equation*}
    \phi_i = \sum_{j \in \widehat{\cI}_{H, i}} \frac{1}{n_{j}} \widehat{\phi}_j.
\end{equation*}
Here, $\widehat{\cI}_{H,i}$ represents the index set for the basis functions in $\{ \widehat{\phi}_j \}_{j \in \widehat{\cI}_H}$, where each $\widehat{\phi}_j$ is associated with a polytope having $x^i$ as its vertex, and $n_j$ denotes the number of vertices of the polytope corresponding to $\widehat{\phi}_j$.
Then it is obvious that~\eqref{coarse_basis1} and~\eqref{coarse_basis2} holds by construction.
Moreover, thanks to the quasi-uniformity of $\cT_H$, one can verify~\eqref{coarse_basis3} using a scaling argument~\cite[Section~3.4]{TW:2005}.

In the following, we provide several examples of $\{ \phi_i \}_{i \in \cI_H}$ for different settings.

\begin{example}
First, we consider a trivial case $m = 1$.
We assume that $\cT_H$ consists of quasi-uniform simplices, and that $\widehat{S}_H$ is the continuous and piecewise linear finite element space.
In this case, we readily observe that the coarse space $V_0$ is given by $V_0 = \widehat{S}_H \cap H_0^1 (\Omega)$, i.e., the piecewise linear finite element space with the homogeneous essential boundary condition.
\end{example}

\begin{example}
\label{Ex:BFS}
In the case when $d = 2$ and $m = 2$, we suppose that $\cT_H$ is given by a rectangular grid and define $\widehat{S}_H$ as the BFS finite element space defined on $\cT_H$.
Then, the coarse space $V_0$, constructed using basis functions of $\widehat{S}_H$, aligns with the one considered in~\cite{Park:2023}.
\end{example}

\begin{figure}
    \centering
    \includegraphics[width=0.97\textwidth]{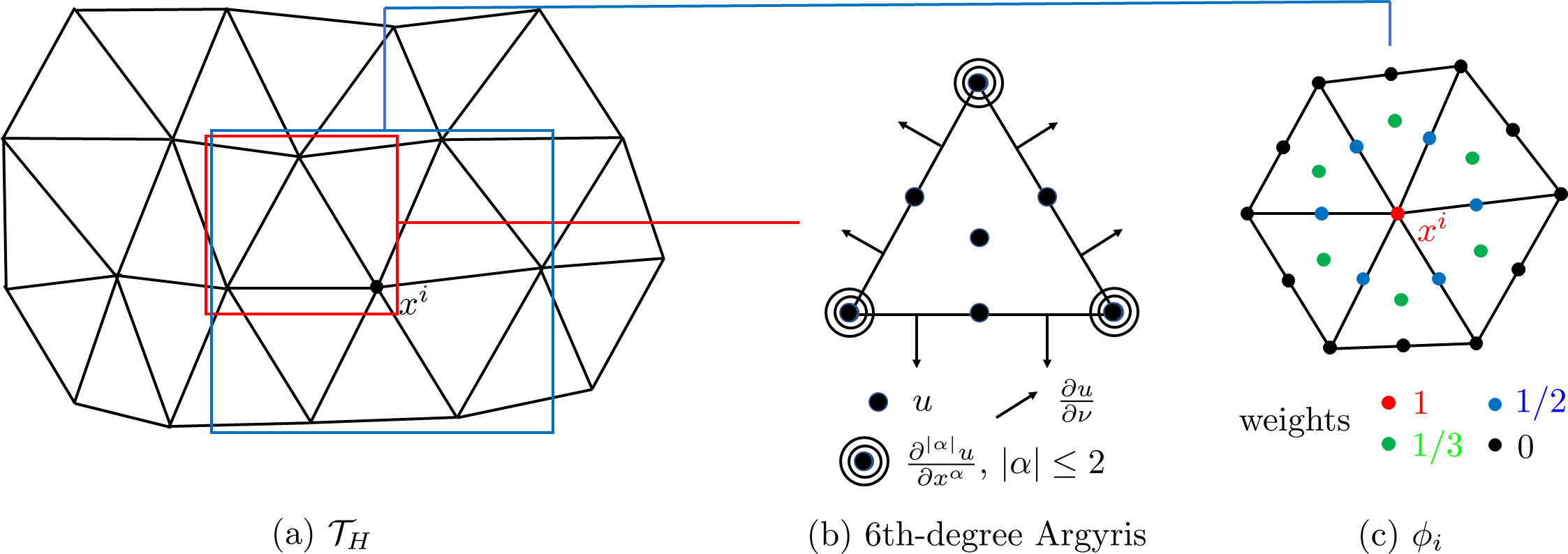}
    \caption{Construction of a collection $\{ \phi_i \}_{i \in \cI_H}$ satisfying~\eqref{coarse_basis} by using the sixth-degree Argyris finite element space~\cite{AFS:1968} on a triangular coarse grid $\cT_H$~($d = 2$, $m = 2$).}
    \label{Fig:coarse_Argyris}
\end{figure}

\begin{example}
We again consider the case where $d = 2$ and $m = 2$.
Let $\widehat{S}_H$ denote the sixth-degree Argyris finite element space~\cite{AFS:1968}, defined on a triangulation $\cT_H$ comprising quasi-uniform triangles, as illustrated in \cref{Fig:coarse_Argyris}(a, b).
The sixth-degree Argyris elements include degrees of freedom for function evaluation at vertices, midpoints of edges, and centers of triangles.
Within each $\omega_i$, $i \in \cI_H$, the coarse basis function $\phi_i \in V_0$ is constructed as a linear combination of global $\widehat{S}_H$-basis functions associated with function evaluation adjacent to $x^i$.
Specifically, the weights are assigned as follows: $1$ for the basis function at $x^i$, $1/2$ for edge basis functions, and $1/3$ for center basis functions; see \cref{Fig:coarse_Argyris}(c).

It is worth noting that this example can be generalized to arbitrary values of $d$ and $m$, as the sixth-degree Argyris element is a specific instance of the general construction of conforming finite elements introduced in~\cite{HLW:2023}.
\end{example}

\begin{figure}
    \centering
    \includegraphics[width=0.95\textwidth]{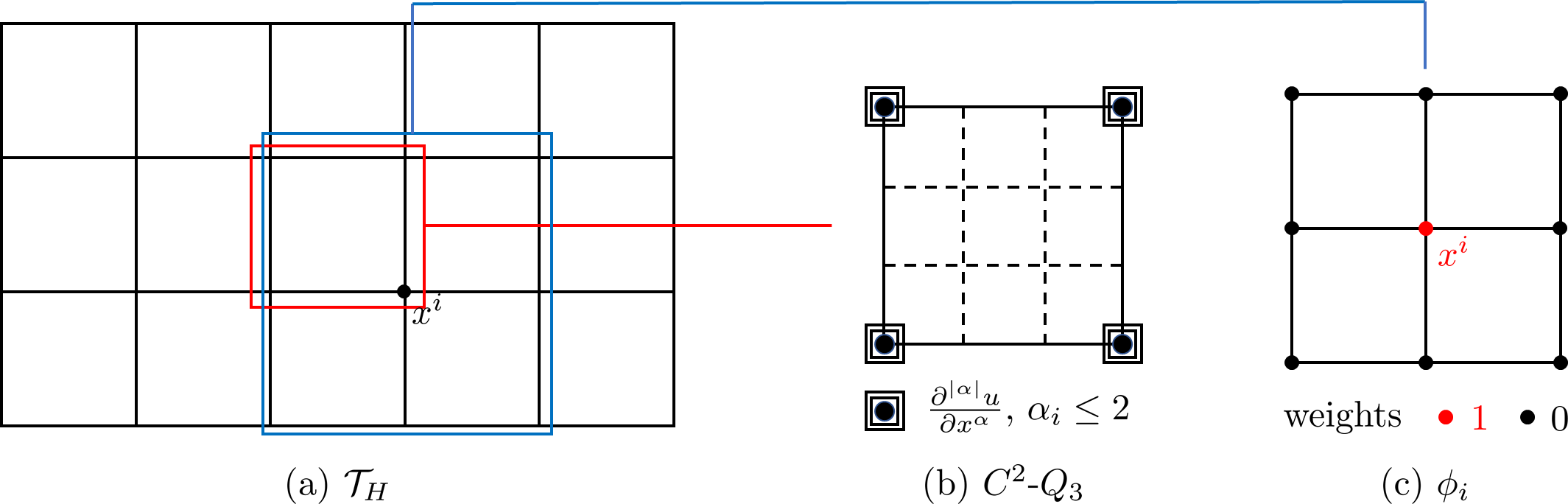}
    \caption{Construction of a collection $\{ \phi_i \}_{i \in \cI_H}$ satisfying~\eqref{coarse_basis} by using the $C^2$-$Q_3$ finite element space~\cite{HZ:2015} on a rectangular coarse grid $\cT_H$~($d = 2$, $m = 3$).}
    \label{Fig:coarse_C2Q3}
\end{figure}

\begin{example}
\label{Ex:C2Q3}
We consider the case where $d = 2$ and $m =3$.
If $\cT_H$ is given by a rectangular grid~(see \cref{Fig:coarse_C2Q3}(a)), we can define $\widehat{S}_H$ as the conforming $C^2$-$Q_3$ finite element space proposed in~\cite{HZ:2015}.
Since all degrees of freedom of the $C^2$-$Q_3$ element corresponding to function evaluation are located at vertices~(see \cref{Fig:coarse_C2Q3}(b)), the coarse basis function $\phi_i \in V_0$, $i \in \cI_H$, simply agrees with the global $C^2$-$Q_3$-basis function at $x^i$, as depicted in \cref{Fig:coarse_C2Q3}(c).

This example can also be extended to arbitrary values of $d$ and $m$ by utilizing the construction of conforming finite elements on rectangular grids introduced in~\cite{HZ:2015}.
\end{example}

\section{Applications}
\label{Sec:Applications}
In this section, we present applications of the two-level additive Schwarz preconditioners proposed in this paper to various finite element discretizations of the $2m$th-order elliptic problem~\eqref{model_cont}.
We also provide numerical results to support our theoretical findings.
All codes used in our numerical experiments were implemented using MATLAB R2022b and executed on a desktop computer equipped with an AMD Ryzen 5 5600X CPU~(3.7GHz, 6C), 40GB RAM, and the Windows 10 Pro operating system.

\begin{figure}
    \centering
    \includegraphics[width=0.83\textwidth]{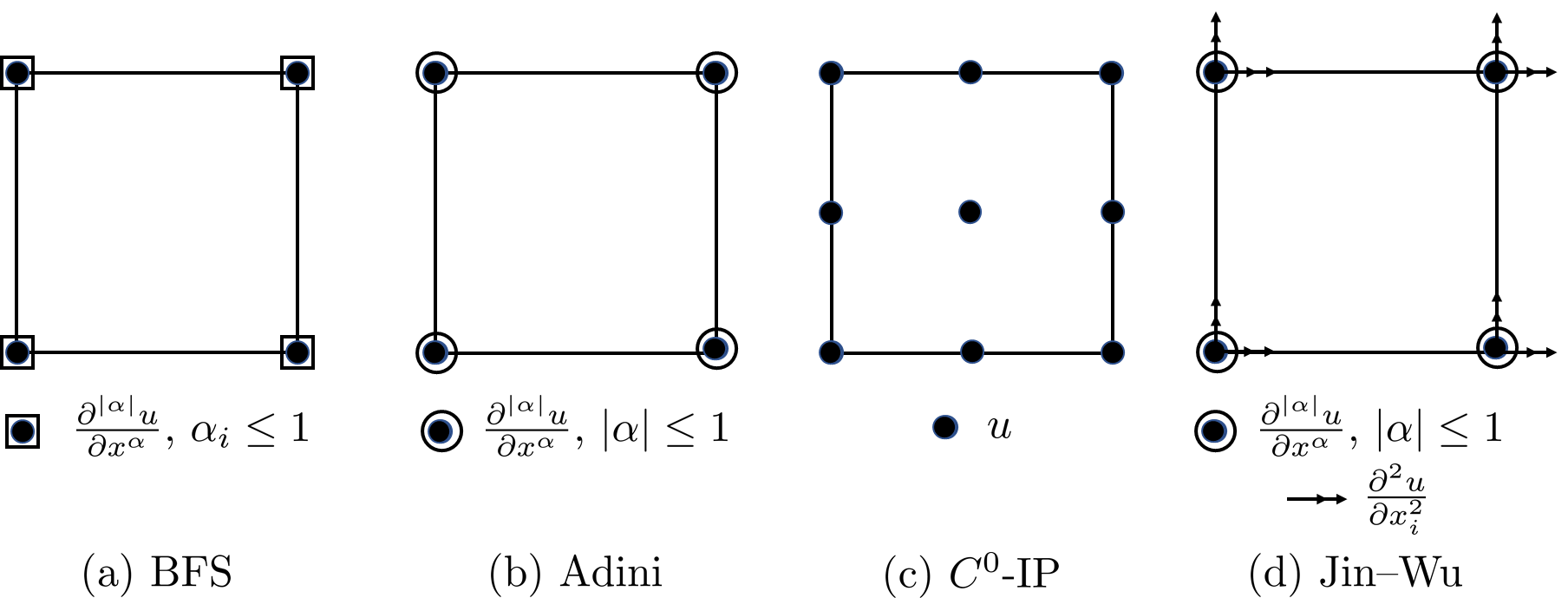}
    \caption{Reference elements of finite element discretizations used in \cref{Sec:Applications}: \textbf{(a)}~Bogner--Fox--Schmit~\cite{BFS:1965,Valdman:2020}, \textbf{(b)}~Adini~\cite{HS:2013,LL:1975}, \textbf{(c)} $C^0$ interior penalty~\cite{BW:2005,EGHLMT:2002}, and \textbf{(d)} Adini-type sixth-order nonconforming elements proposed by Jin and Wu~\cite{JW:2023}.}
    \label{Fig:reference_elements}
\end{figure}

\subsection{Fourth-order conforming elements}
First, we consider a BFS~\cite{BFS:1965,Valdman:2020,WSX:2007} finite element discretization of the fourth-order problem in two dimensions~($m = 2$, $d = 2$).
That is, the finite element space $S_h$ is given by
\begin{align*}
S_T &= \mathbb{Q}_3 (T), \quad T \in \cT_h, \\
S_h &= \left\{ u \in H_0^2 (\Omega) : u|_T \in S_T \text{ for all } T \in \cT_h \right\},
\end{align*}
where $\mathbb{Q}_3$ denotes the tensor product of $\mathbb{P}_3$. 
The reference element for the BFS element is depicted in \cref{Fig:reference_elements}(a).
Since we adopt a conforming element, we set the bilinear form $a_h (\cdot, \cdot)$ in the discrete problem~\eqref{model_FEM} as the same on as in the continuous problem~\eqref{model_cont}, i.e.,
\begin{equation*}
a_h(u,v) = \int_{\Omega} \nabla^2 u : \nabla^2 v \,dx,
\quad u, v \in S_h.
\end{equation*}
We construct the coarse space $V_0$ as explained in \cref{Ex:BFS}, with $\widehat{S}_H$ given by the BFS finite element space on the coarse mesh $\cT_H$.

To apply the convergence theory developed in this paper, we need to verify \cref{Ass:FEM,Ass:local,Ass:coarse}.
We set
\begin{equation*}
    \| u \|_h = \sqrt{a_h (u, u)} = | u |_{H^2 (\Omega)},
    \quad u \in H_0^2 (\Omega).
\end{equation*}
In addition, we set $\tS_h = S_h$ and define the enriching operator $E_h \colon S_h \rightarrow \tS_h$ as the identity operator.
Then \cref{Ass:FEM} obviously holds.
Verification of \cref{Ass:local} can be done by combining the argument given in~\cite[Theorem~2.3]{Zhang:1996} with \cref{Lem:trace}.
Finally, thanks to the polynomial approximation theory in Sobolev spaces~(see, e.g.,~\cite{BS:2008,Ciarlet:2002}), \cref{Ass:coarse} can be achieved by setting the coarse prolongation operator $R_0^T$ by the nodal interpolation operator.
Consequently, \cref{Thm:kappa} implies that the convergence rate of an iterative algorithm preconditioned by the two-level additive Schwarz preconditioner is uniformly bounded when $H/\delta$ is fixed.

Now, we present numerical results.
In~\eqref{model_cont} with $m = 2$ and $d = 2$, we set $\Omega = (0,1)^2$ and $f$ such that the exact solution $u$ is given by $u(x_1, x_2) = x_1^2 (1-x_1)^2 \sin^2 \pi x_2$.
We use the same mesh and domain decomposition settings as in~\cite[Section~6]{Park:2023}; let $\cT_H$ be a coarse triangulation of $\Omega$ consisting of $N = 1/H \times 1/H$ square elements, and let $\cT_h$ be a refinement of $\cT_H$ consisting of $1/h \times 1/h$ square elements~($0 < h < H < 1$).
By extending each coarse element in $\cT_H$ to include its surrounding layers of $\cT_h$-elements with width $\delta$, we construct an overlapping domain decomposition $\{ \Omega_k \}_{k=1}^N$ of $\Omega$.

\begin{figure}
\centering
\subfloat[][Comparison]{ \includegraphics[width=0.31\linewidth]{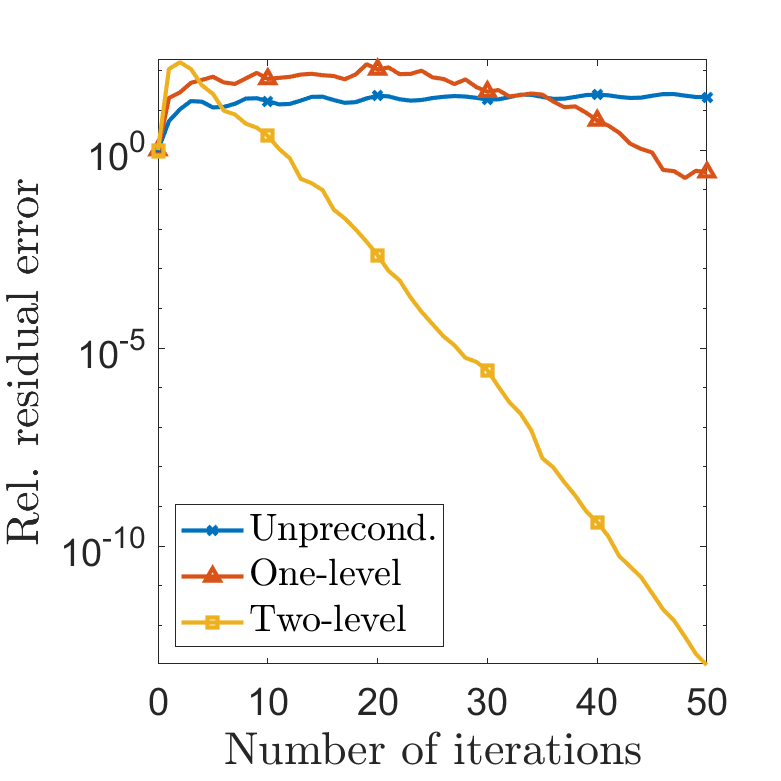} }
\subfloat[][$\delta = 2^1 h$~($H/\delta = 2^3$)]{ \includegraphics[width=0.31\linewidth]{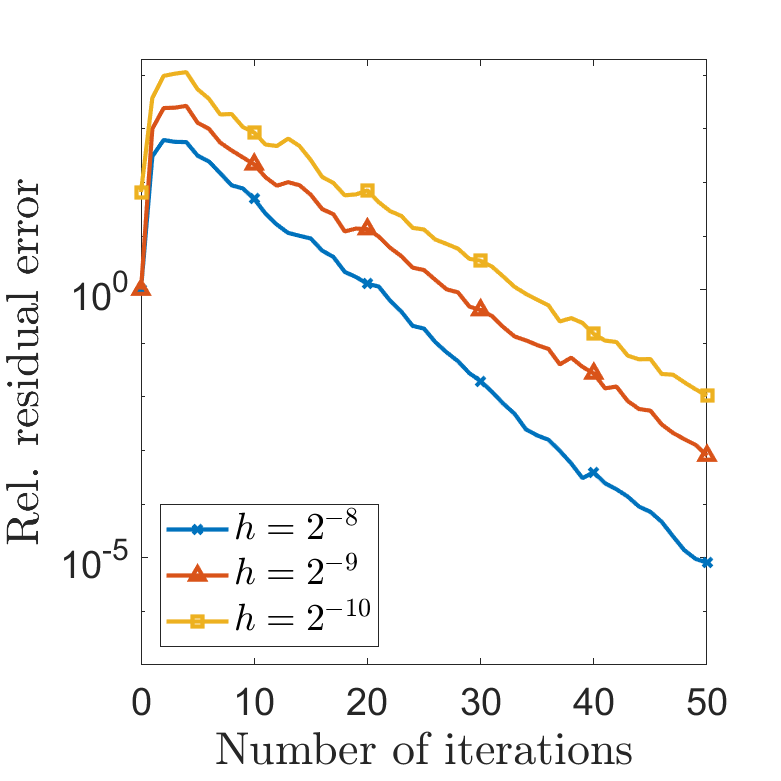} }
\subfloat[][$\delta = 2^2 h$~($H/\delta = 2^2$)]{ \includegraphics[width=0.31\linewidth]{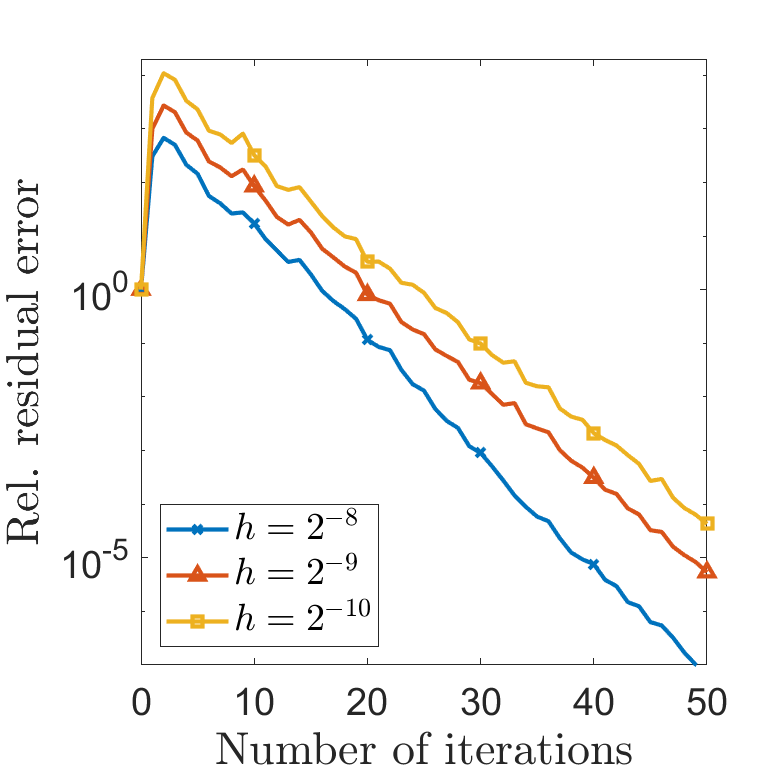} }
\caption{Relative residual error~\eqref{error} of the conjugate gradient method preconditioned by the proposed two-level additive Schwarz preconditioner for solving the fourth-order Bogner--Fox--Schmit finite element discretization.
\emph{\textbf{(a)}} Comparison with the cases of no preconditioner and the one-level preconditioner~($h = 2^{-7}$, $H = 2^{-3}$, $\delta = 2^{-5}$).
\emph{\textbf{(b, c)}} Various mesh sizes $h$~($H/h = 2^4$).}
\label{Fig:BFS}
\end{figure}

To observe the convergence behavior of iterative algorithms, we plot the relative residual error
\begin{equation}
\label{error}
\frac{\| Au^{(n)} - f \|_{\ell^2}}{\| Au^{(0)} - f \|_{\ell^2}}
\end{equation}
at each iteration, where $A$ and $f$ were given in~\eqref{model_linear}, and $n$ denotes the iteration count.
In \cref{Fig:BFS}(a), the convergence curves of the conjugate gradient method without preconditioner and with one- and two-level additive Schwarz preconditioners are plotted.
We readily observe that the convergence rate of the two-level preconditioner is much faster than the other ones, which highlights the improvement made by adding the coarse space.
In \cref{Fig:BFS}(b, c), the convergence curves for the two-level additive Schwarz preconditioners are presented, under varying $h$, $H$ and $\delta$, keeping $H/h = 2^4$.
We observe that the convergence curves for the cases $h = 2^{-9}$ and $h = 2^{-10}$ are parallel, which implies that the convergence rate do not deteriorate even if we decrease the mesh size $h$.
This numerically verifies \cref{Thm:kappa}, which states that the condition number of $M^{-1} A$ is uniformly bounded if $H / \delta$ is kept constant.

\begin{remark}
\label{Rem:jump}
In \cref{Fig:BFS}(b, c), we observe increases in the relative residual errors~\eqref{error} during the initial iterations of the conjugate gradient method preconditioned by the proposed two-level additive Schwarz preconditioner.
Furthermore, the magnitude of these jumps increases as $h$ decreases.
Consequently, the number of iterations required to achieve a certain level of relative residual error slightly increases as $h$ decreases, even though the convergence curves remain parallel.
A rigorous analysis of these jumps in the relative residual error remains an open problem.
\end{remark}

\subsection{Fourth-order nonconforming elements}
As the second example, we consider a nonconforming finite element method; we assume that the fourth-order problem in two dimensions~($m = 2$, $d = 2$) is discretized by the Adini element~\cite{HS:2013,LL:1975,WSX:2007}, in which the finite element space $S_h$ is given by
\begin{align*}
S_T &= \mathbb{P}_3 (T) \oplus \operatorname{span} \left\{  x_1^3 x_2, x_1 x_2^3 \right\}, \quad T \in \cT_h, \\
S_h &= \bigg\{ u \in L^2 (\Omega) : u|_T \in S_T \text{ for all } T \in \cT_h,\text{ } u, \frac{\partial u}{\partial x_1}, \text{ and } \frac{\partial u}{\partial x_2}  \text{ are continuous} \\
&\quad\quad \text{ at the vertices of } \cT_h \text{ and vanish at the vertices along } \partial \Omega \bigg\}.
\end{align*}
The Adini reference element is depicted in \cref{Fig:reference_elements}(b).
As in~\cite{Brenner:1996}, we define the bilinear form $a_h (\cdot, \cdot)$ as a broken version of $a( \cdot, \cdot)$ as follows:
\begin{equation}
\label{a_h_Adini}
a_h (u,v) = \sum_{T \in \cT_h} \int_T \nabla^2 u : \nabla^2 v \,dx,
\quad u,v \in S_h.
\end{equation}
We define the coarse space $V_0$ in the same manner as in the previous example, i.e., by setting $\widehat{S}_h$ to be the BFS finite element space on $\cT_H$.

Next, we verify \cref{Ass:FEM,Ass:local,Ass:coarse} to apply the convergence theory presented in this paper.
We set
\begin{equation*}
    \| u \|_h = \sqrt{a_h (u, u)},
    \quad u \in S_h + H_0^2 (\Omega).
\end{equation*}
In addition, we define $\widetilde{S}_h$ as the BFS finite element space defined on $\mathcal{T}_h$, and $E_h \colon S_h \rightarrow \widetilde{S}_h$ as the enriching operator introduced in~\cite[Equation~(5.4)]{Brenner:1996}.
Then, invoking~\cite[Lemma~5.1]{Brenner:1996} implies that \cref{Ass:FEM} holds.
For \cref{Ass:local}, one may refer to~\cite[Equation~(8.16)]{Brenner:1996}.
Finally, if we set the coarse prolongation operator $R_0^T$ by the nodal interpolation operator, then \cref{Ass:coarse} can be proven by using the standard polynomial approximation theory~\cite{BS:2008,Ciarlet:2002}.
Therefore, \cref{Thm:kappa} ensures that the convergence rate of an iterative algorithm preconditioned by the two-level additive Schwarz preconditioner is uniformly bounded when $H/\delta$ is fixed.

\begin{figure}
\centering
\subfloat[][Comparison]{ \includegraphics[width=0.31\linewidth]{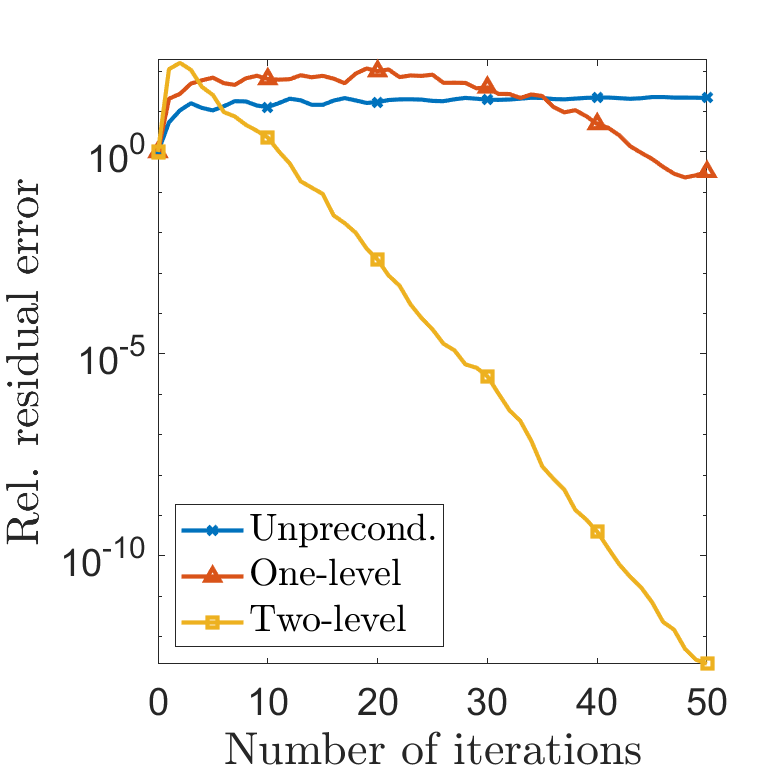} }
\subfloat[][$\delta = 2^1 h$~($H/\delta = 2^3$)]{ \includegraphics[width=0.31\linewidth]{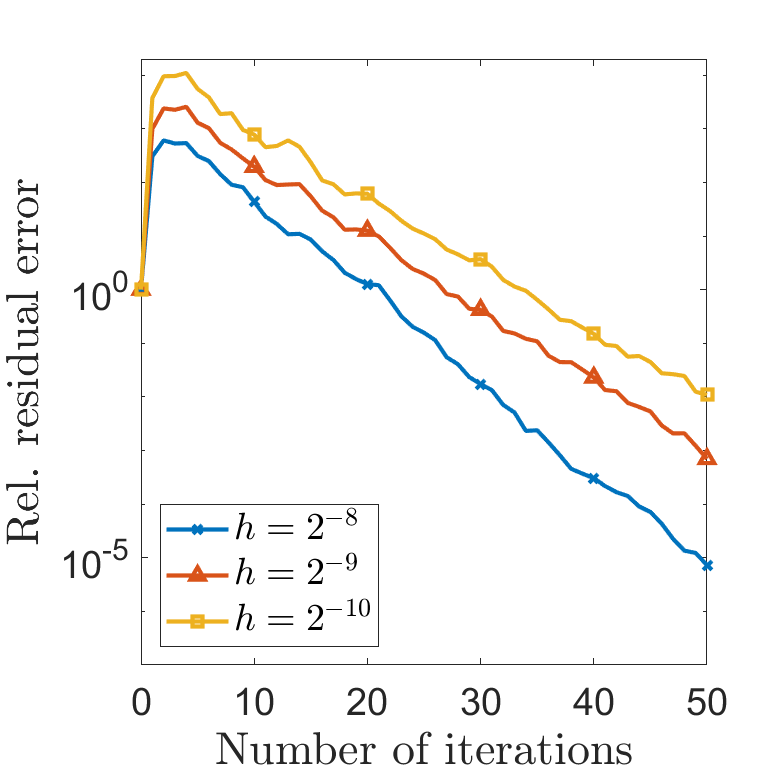} }
\subfloat[][$\delta = 2^2 h$~($H/\delta = 2^2$)]{ \includegraphics[width=0.31\linewidth]{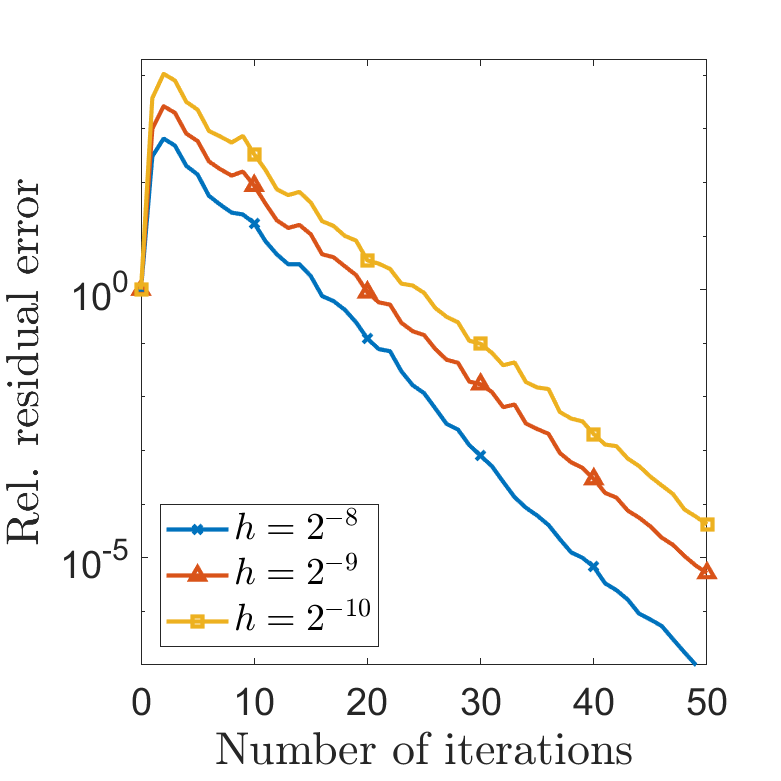} }
\caption{Relative residual error~\eqref{error} of the conjugate gradient method preconditioned by the proposed two-level additive Schwarz preconditioner for solving the fourth-order Adini finite element discretization.
\emph{\textbf{(a)}} Comparison with the cases of no preconditioner and the one-level preconditioner~($h = 2^{-7}$, $H = 2^{-3}$, $\delta = 2^{-5}$).
\emph{\textbf{(b, c)}} Various mesh sizes $h$~($H/h = 2^4$).}
\label{Fig:Adini}
\end{figure}

We conduct numerical experiments using the same configurations as in the previous example.
The numerical results for the Adini element are presented in \cref{Fig:Adini}.
We can deduce the same conclusion as in the BFS case.
We observe a notable enhancement in the convergence behavior of the conjugate gradient method due to the presence of the coarse space, as depicted in \cref{Fig:Adini}(a).
Furthermore, in \cref{Fig:Adini}(b, c), we note that the convergence rate remains unaffected even as $h$ decreases, provided $H / \delta$ is held constant.
This observation verifies \cref{Thm:kappa} and highlights the numerical scalability of the proposed two-level preconditioner when applied to the Adini element.

\subsection{Fourth-order interior penalty methods}
The coarse space construction proposed in this paper is also applicable for discontinuous Galerkin-type discretizations.
As an illustrative example, we consider the fourth-order problem in two dimensions~($m = 2$, $d = 2$) discretized by the $C^0$ interior penalty method introduced in~\cite{BS:2005,EGHLMT:2002}.
That is, we have
\begin{align*}
S_T &= \mathbb{Q}_2 (T) , \quad T \in \cT_h, \\
S_h &= \left\{ u \in C^0 (\Omega) : u|_T \in S_T \text{ for all } T \in \cT_h,\text{ } u = 0 \text{ on } \partial \Omega \right\}.
\end{align*}
The reference element for the $C^0$ interior penalty method is depicted in \cref{Fig:reference_elements}(c).
The bilinear form $a_h (\cdot, \cdot)$ is given by
\begin{multline}
\label{a_h_C0IP}
a_h (u,v) = \sum_{T \in \cT_h} \int_T \nabla^2 u : \nabla^2 v \,dx + \sum_{e \in \mathcal{E}_h} \int_e \left( \left\{\!\!\!\left\{ \frac{\partial^2 u}{\partial \nu^2} \right\}\!\!\!\right\} \left[\!\!\left[\frac{\partial v}{\partial \nu } \right]\!\!\right] + \left\{\!\!\!\left\{ \frac{\partial^2 v}{\partial \nu^2} \right\}\!\!\!\right\} \left[\!\!\left[\frac{\partial u}{\partial \nu } \right]\!\!\right] \right) \,ds \\
+ \sum_{e \in \mathcal{E}_h} \frac{\eta}{|e|} \int_e \left[\!\!\left[\frac{\partial u}{\partial \nu } \right]\!\!\right] \left[\!\!\left[\frac{\partial v}{\partial \nu } \right]\!\!\right] \,ds,
\end{multline}
where $\mathcal{E}_h$ denotes the set of all edges of $\cT_h$, and $\eta$ is a positive penalty parameter.
The jumps $[\![ \cdot ]\!]$ and the averages $\{\!\!\{ \cdot \}\!\!\}$ in~\eqref{a_h_C0IP} are defined as follows.
For an interior edge $e$ in $\mathcal{E}_h$ shared by two elements $T_+$ and $T_-$ in $\cT_h$, let $\nu_e$ be the unit normal vector of $e$ that points from $T_-$ to $T_+$.
Then we define
\begin{equation*}
\left[\!\!\left[\frac{\partial u}{\partial \nu } \right]\!\!\right] = \frac{\partial u|_{T_+}}{\partial \nu_e} - \frac{\partial u|_{T_-}}{\partial \nu_e}, \quad
\left\{\!\!\!\left\{ \frac{\partial^2 u}{\partial \nu^2} \right\}\!\!\!\right\} = \frac{1}{2 } \left( \frac{\partial^2 u |_{T_+}}{\partial \nu_e^2} + \frac{\partial^2 u|_{T_-}}{\partial \nu_e^2} \right).
\end{equation*}
For a boundary edge $e$ on $\partial \Omega$, let $\nu_e$ denote the outward unit normal vector of $e$, and we define
\begin{equation*}
\left[\!\!\left[\frac{\partial u}{\partial \nu } \right]\!\!\right] = - \frac{\partial u}{\partial \nu_e}, \quad
\left\{\!\!\!\left\{ \frac{\partial^2 u}{\partial \nu^2} \right\}\!\!\!\right\} = \frac{\partial^2 u}{\partial \nu_e^2}.
\end{equation*}
If $\eta$ is sufficiently large, then $a_h (\cdot, \cdot)$ becomes coercive~\cite{BS:2005}.

We set
\begin{equation*}
\| u \|_h = \sqrt{\sum_{T \in \cT_h} | u |_{H^2 (T)}^2 + \sum_{e \in \mathcal{E}_h} \frac{1}{|e|} \left\| \frac{\partial u}{\partial \nu} \right\|_{L^2 (e)}^2},
\quad u \in S_h + H_0^2 (\Omega).
\end{equation*}
The conforming space $\widetilde{S}_h$ and the enriching operator $E_h \colon S_h \rightarrow \widetilde{S}_h$ are defined as in~\cite{BS:2005}.
Again, the coarse space $V_0$ is defined as the one given in \cref{Ex:BFS}, and the coarse prolongation operator $R_0^T$ is defined as the nodal interpolation operator.
Then \cref{Ass:FEM,Ass:local,Ass:coarse} can be verified by referring to~\cite[Lemma~3.4, Equation~(5.8), and Lemma~3.2]{BW:2005}, respectively.

\begin{figure}
\centering
\subfloat[][Comparison]{ \includegraphics[width=0.31\linewidth]{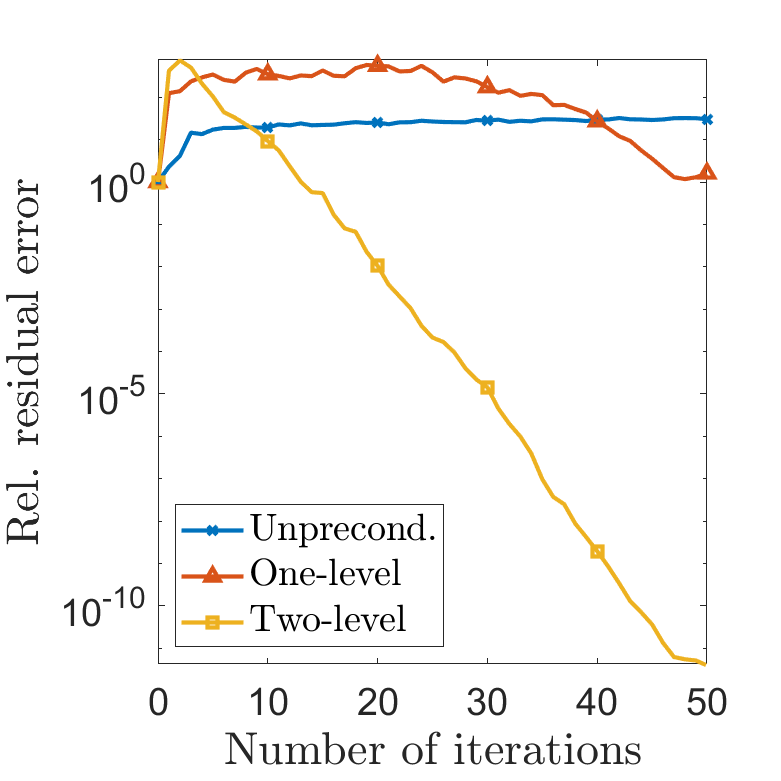} }
\subfloat[][$\delta = 2^1 h$~($H/\delta = 2^3$)]{ \includegraphics[width=0.31\linewidth]{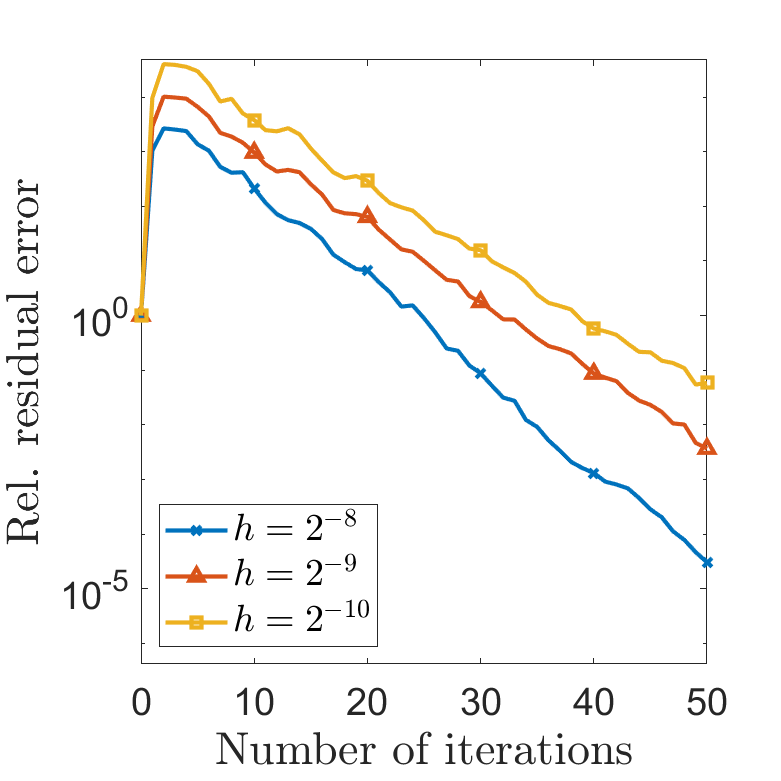} }
\subfloat[][$\delta = 2^2 h$~($H/\delta = 2^2$)]{ \includegraphics[width=0.31\linewidth]{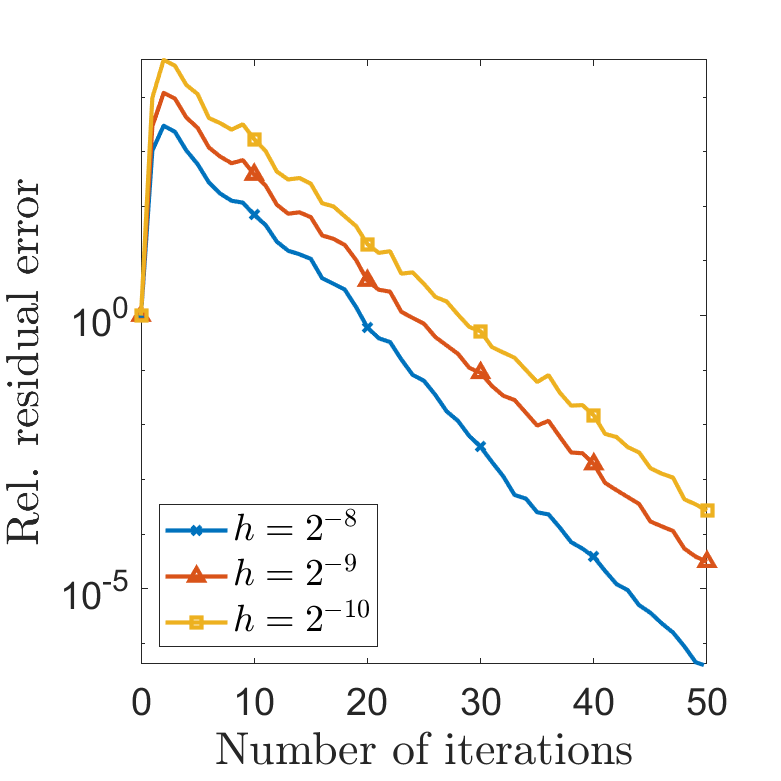} }
\caption{Relative residual error~\eqref{error} of the conjugate gradient method preconditioned by the proposed two-level additive Schwarz preconditioner for solving the fourth-order $C^0$ interior penalty method~($\eta =5$).
\emph{\textbf{(a)}} Comparison with the cases of no preconditioner and the one-level preconditioner~($h = 2^{-7}$, $H = 2^{-3}$, $\delta = 2^{-5}$).
\emph{\textbf{(b, c)}} Various mesh sizes $h$~($H/h = 2^4$).}
\label{Fig:C0IP}
\end{figure}

Under the same configurations as in the previous examples, we present numerical results for the fourth-order $C^0$ interior penalty method in \cref{Fig:C0IP}.
Consistently with the cases of BFS and Adini elements, we observe the numerical scalability of the proposed two-level additive Schwarz preconditioner in the sense that the convergence rate does not deteriorate with decreasing $h$ as long as $H/ \delta$ is fixed.

\subsection{Sixth-order nonconforming elements}
To further validate the applicability of the proposed two-level additive Schwarz preconditioner for higher-order problems, we consider the sixth-order instance~($m = 3$, $d = 2$) of~\eqref{model_FEM}.
We assume that the problem is discretized by the Adini-type finite element proposed by Jin and Wu~\cite{JW:2023}, which is given by
\begin{align*}
S_T &= \operatorname{span} \left\{ \mathbb{Q}_1 (T) \cdot \left\{ 1,  x_1^2, x_2^2, x_1^4, x_2^4 \right\} \right\}, \quad T \in \cT_h, \\
S_h &= \bigg\{ u \in L^2 (\Omega) : u|_T \in S_T \text{ for all } T \in \cT_h,\text{ } u, \frac{\partial u}{\partial x_1}, \frac{\partial u}{\partial x_2}, \frac{\partial^2 u}{\partial x_1^2}, \text{ and } \frac{\partial^2 u}{\partial x_2^2}  \text{ are} \\
&\quad\quad \text{ continuous at the vertices of } \cT_h \text{ and vanish at the vertices along } \partial \Omega \bigg\},
\end{align*}
where $\cdot$ denotes the collection of entrywise products of two sets.
The Jin--Wu reference element is depicted in \cref{Fig:reference_elements}(d).
As in~\cite{JW:2023}, we define the bilinear form $a_h (\cdot, \cdot)$ as
\begin{equation*}
a_h (u,v) = \sum_{T \in \mathcal{T}_h} \int_T \nabla^3 u : \nabla^3 v \,dx,
\quad u,v \in S_h.
\end{equation*}
To construct the coarse space $V_0$, we set $\widehat{S}_H$ as the $C^2$-$Q_3$ finite element space~\cite{HZ:2015} presented in \cref{Ex:C2Q3}, and follow the procedure described in \cref{Sec:Universal}.

To verify \cref{Ass:FEM,Ass:local,Ass:coarse} for the Jin--Wu element, similar to the case of the fourth-order Adini element, we set
\begin{equation*}
\| u \|_h = \sqrt{a_h (u,u)}, \quad u \in S_h + H_0^3 (\Omega).
\end{equation*}
Then \cref{Ass:FEM} is satisfied with the conforming space $\widetilde{S}_h$ and the enriching operator $E_h \colon S_h \rightarrow \widetilde{S}_h$ defined in the same manner as in~\cite[Lemma~3.1]{JW:2023}.
Moreover, by using a similar argument as in~\cite[Lemma~7.3]{Brenner:1996} and \cref{Lem:trace}, we can verify \cref{Ass:local} without difficulty.
Finally, the standard polynomial approximation theory~\cite{BS:2008,Ciarlet:2002} implies that \cref{Ass:coarse} holds with the coarse prolongation operator $R_0^T$ given by the nodal interpolation operator~(cf.~\cite[Theorem~3.1]{JW:2023}).

In numerical experiments, we set $\Omega = (0,1)^2$ and $f$, such that the exact solution is given by $u(x_1, x_2) = x_1^3 (1-x_1)^3 \sin^3 \pi x_2$, in~\eqref{model_cont} with $m = 3$ and $d = 2$.
We maintain the same mesh and domain decomposition settings as those utilized in the previous examples.

\begin{figure}
\centering
\subfloat[][Comparison]{ \includegraphics[width=0.31\linewidth]{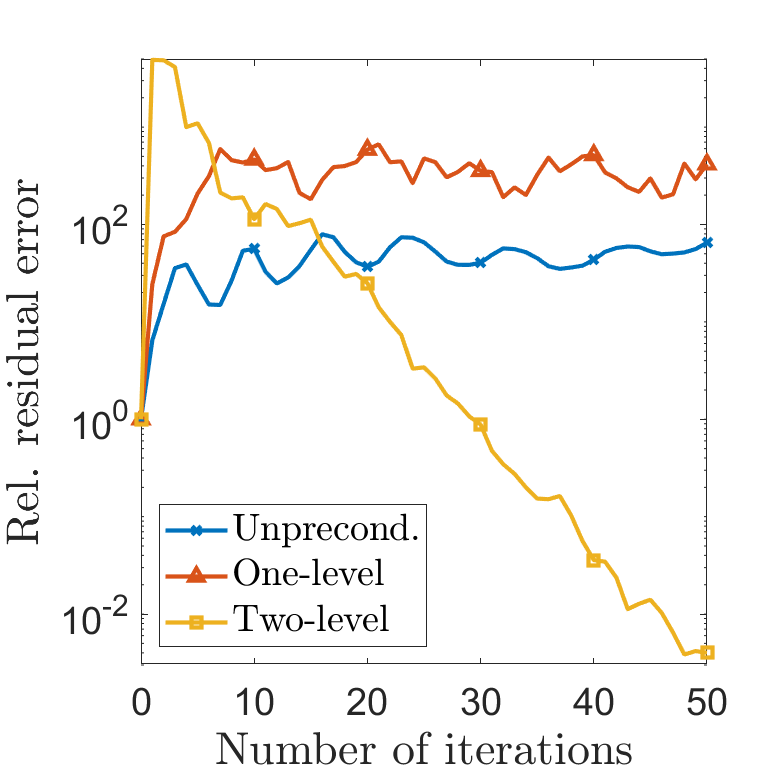} }
\subfloat[][$\delta = 2^1 h$~($H/\delta = 2^3$)]{ \includegraphics[width=0.31\linewidth]{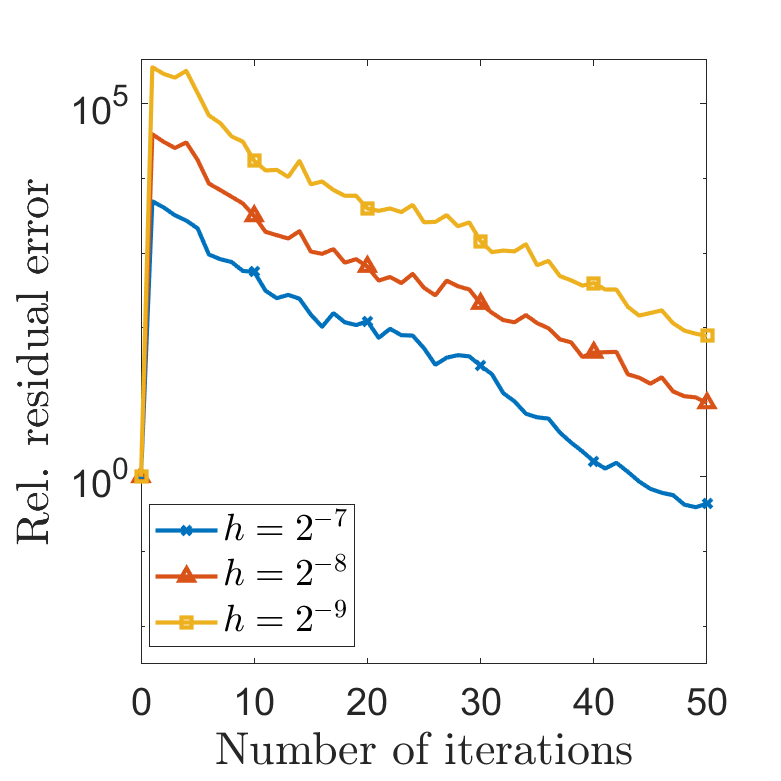} }
\subfloat[][$\delta = 2^2 h$~($H/\delta = 2^2$)]{ \includegraphics[width=0.31\linewidth]{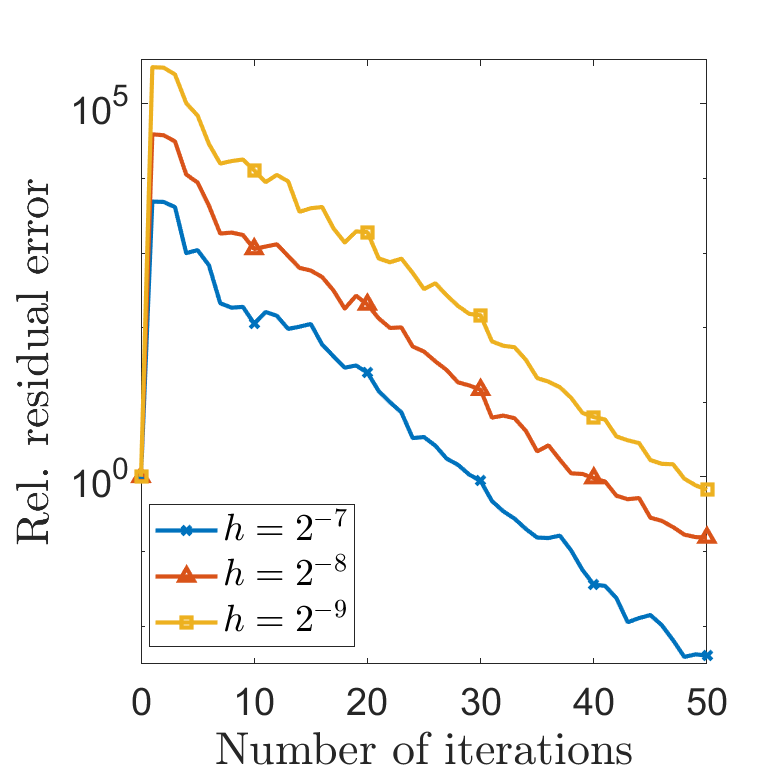} }
\caption{Relative residual error~\eqref{error} of the conjugate gradient method preconditioned by the proposed two-level additive Schwarz preconditioner for solving the sixth-order Jin--Wu finite element discretization.
\emph{\textbf{(a)}} Comparison with the cases of no preconditioner and the one-level preconditioner~($h = 2^{-7}$, $H = 2^{-3}$, $\delta = 2^{-5}$).
\emph{\textbf{(b, c)}} Various mesh sizes $h$~($H/h = 2^4$).}
\label{Fig:JW}
\end{figure}

In \cref{Fig:JW}(a), we present the convergence curves of the conjugate gradient method both without preconditioner and with one- and two-level additive Schwarz preconditioners.
Analogous to the fourth-order cases, we observe a significant enhancement in the convergence rate upon integrating the proposed coarse space into the additive Schwarz preconditioner.
Moreover, \cref{Fig:JW}(b, c) numerically shows that the proposed two-level additive Schwarz preconditioner is scalable when applied to the Jin--Wu element.
The convergence rate of the preconditioned method is uniformly bounded as $h$ decreases keeping $H/\delta$ constant.

\section{Conclusion}
\label{Sec:Conclusion}
In this paper, we proposed a novel construction of two-level overlapping Schwarz preconditioners, which can be applied to any finite element discretization of $2m$th-order elliptic problems that satisfies typical assumptions.
For any finite element discretization, a coarse space can be constructed in a unified way based on a conforming finite element space defined on a coarse mesh.
We proved that the condition number of the preconditioned stiffness matrix achieves a bound that depends on $H/\delta$ only, which implies the scalability of iterative methods.
We presented applications of the proposed two-level overlapping Schwarz preconditioners to various finite elements, such as the BFS, Adini, fourth-order $C^0$ interior penalty, and Jin--Wu elements, and consistently observed an improvement of the convergence rate and the scalability.

We conclude the paper with a discussion of several intriguing topics for future research.
While the construction of coarse spaces proposed in this paper utilizes conforming finite elements, it is interesting to further investigating whether nonconforming finite elements with fewer degrees of freedom could be employed, as the computational cost for coarse problems remains a bottleneck in parallel computation and it is desirable to reduce it~\cite{DW:2017}.
Meanwhile, since the heterogeneity of problems is an important issue in many applications~(see, e.g.,~\cite{EGH:2013,HW:1997,LMP:2022}), it is crucial to consider how to extend the proposed construction of coarse spaces to problems with heterogeneous coefficients.
In addition, extending our approach to higher-order variants of nonlinear and nonsmooth problems~\cite{LP:2022,Park:2023,Park:2024} is an interesting topic for future research.

\appendix
\section{Friedrichs inequality for \texorpdfstring{$H^m$}{Hm}-functions}
In this appendix, for the sake of completeness, we provide the Friedrichs inequality for $H^m$-functions used throughout the paper along with its proof.
We first show that any polynomial defined on a bounded polyhedral domain $D$ and satisfying a homogeneous boundary condition on a portion of $\partial D$ must be identically zero.

\begin{lemma}
\label{Lem:polynomial_BC}
Let $D \subset \mathbb{R}^d$ ($d \in \mathbb{Z}_{> 0 })$ be a bounded polyhedral domain, and let $\Gamma \subset \partial D$ have nonvanishing $(d-1)$-dimensional measure.
If an $m$th-degree polynomial $p \in \mathbb{P}_{m} (D)$ satisfies
\begin{equation*}
p = \frac{\partial p}{\partial \nu} = \dots = \frac{\partial^{m} p}{\partial \nu^{m}} = 0 \quad \text{ on } \Gamma,
\end{equation*}
where $\nu$ is the outward unit normal vector field along $\Gamma$, then $p = 0$.
\end{lemma}
\begin{proof}
There exists at least one hyperface $F$ of $\partial D$ such that $F \cap \Gamma$ has nonvanishing $(d-1)$-dimensional measure.
Since $p$ vanishes on $F \cap \Gamma$ with nonvanishing measure, it also vanishes on the entire hyperplane $L$ containing $F \cap \Gamma$~\cite{CT:2005}.
We perform an affine change of coordinates such that the hyperplane $L (\hat{x}, x_d) = 0$ is the $\hat{x}$-axis, and $\nu = (0, \dots, 0, 1)$,  where $\hat{x} = (x_1, \dots, x_{d-1})$.
Since $p = 0$ on the $\hat{x}$-axis, by invoking~\cite[Lemma~3.1.10]{BS:2008}, we have $p = x_d p_1$ for some polynomial $p_1$.
Moreover, since $\frac{\partial p}{\partial \nu} = \frac{\partial p}{\partial x_d} = p_1 + x_d \frac{\partial p_1}{\partial x_d}$ vanishes on the $\hat{x}$-axis, invoking~\cite[Lemma~3.1.10]{BS:2008} once again yields that $p_1 = x_d p_2$ for some polynomial $p_2$, i.e., $p = x_d^2 p_2$.
Repeating this procedure, we deduce that $p = x_d^{m+1} p_{m+1}$ for some polynomial $p_{m+1}$.
But since $p \in \mathbb{P}_m (D)$, $p_{m+1}$ must be identically zero, which completes the proof.
\end{proof}

Now, we present the Friedrichs inequality for $H^m$-functions in \cref{Thm:Friedrichs}, which can be proven by using \cref{Lem:polynomial_BC} and following an argument similar to that described in~\cite[Section~1.1.8]{Necas:2012}.
In the following, $\| \cdot \|_{H^m (D)}$ denotes a scaled $H^m$-norm given by
\begin{equation}
\label{scaled_norm}
    \| u \|_{H^m (D)} = \left( \| u \|_{L^2 (D)}^2 + \sum_{j=1}^m H^{2j} | u |_{H^j (D)}^2 \right)^{\frac{1}{2}},
    \quad u \in H^m (D),
\end{equation}
where $H = \operatorname{diam} D$.
This scaled norm is commonly used in the analysis of domain decomposition methods; see, e.g.,~\cite[Equation~(4.4)]{TW:2005}.

\begin{theorem}
\label{Thm:Friedrichs}
Let $D \subset \mathbb{R}^d$ ($d \in \mathbb{Z}_{> 0 })$ be a bounded polyhedral domain, and let $\Gamma \subset \partial D$ have nonvanishing $(d-1)$-dimensional measure.
Then there exist a positive constant $C$, depending only on the shapes of $D$ and $\Gamma$, such that
\begin{equation*}
    \| u \|_{H^m (D)}^2
    \leq C \left( H^{2m} | u|_{H^m (D)}^2 + \sum_{j=0}^{m-1} H^{2j + 1} \left\| \frac{\partial^j u}{\partial \nu^j} \right\|_{L^2 (\Gamma)}^2 \right),
    \quad u \in H^m (D),
\end{equation*}
where $H = \operatorname{diam} D$, $\| \cdot \|_{H^m (D)}$ is the scaled $H^m$-norm given in~\eqref{scaled_norm}, and $\nu$ is the outward unit normal vector field along $\Gamma$.
\end{theorem}
\begin{proof}
It suffices to take $H = 1$, as the general case then follows by a standard scaling argument~\cite[Section~3.4]{TW:2005}.
Throughout this proof, let $C$ denote the generic positive constant that depends only on the shapes of $D$ and $\Gamma$.
We define
\begin{equation*}
\| u \|_{\text{RHS}} = \left( | u |_{H^m (D)}^2 + \sum_{j=0}^{m-1} \left\| \frac{\partial^j u}{\partial \nu^j} \right\|_{L^2 (\Gamma)}^2 \right)^{\frac{1}{2}},
\quad u \in H^m (D).
\end{equation*}
Then \cref{Lem:polynomial_BC} implies that $\| \cdot \|_{\text{RHS}}$ is indeed a norm.
Moreover, by the trace theorem, we have $\| \cdot \|_{\text{RHS}} \leq C \| \cdot \|_{H^m (D)}$.
Hence, if we show that $H^m (D)$ is a Banach space with the $\| \cdot \|_{\text{RHS}}$-norm, then invoking the open mapping theorem yields $\| \cdot \|_{H^m (D)} \leq C \| \cdot \|_{\text{RHS}}$, which is the desired result.

It remains to prove that $H^m (D)$ is complete with respect to the $\| \cdot \|_{\text{RHS}}$-norm.
Take any Cauchy sequence $\{ u_n \}$ with respect to the $\| \cdot \|_{\text{RHS}}$-norm.
Since $| \cdot |_{H^m (D)} \leq \| \cdot \|_{\text{RHS}}$, a quotient space argument~\cite[Theorem~1.6]{Necas:2012} implies that there exist $(m-1)$th-degree polynomials $\{ p_n \}$ such that $\{ u_n + p_n \}$ converges in the $\| \cdot \|_{H^m (D)}$-norm.
As $\| \cdot \|_{\text{RHS}} \leq C \| \cdot \|_{H^m (D)}$, we readily deduce that $\{ u_n + p_n \}$ converges in the $\| \cdot \|_{\text{RHS}}$-norm, thereby forming a Cauchy sequence.
Consequently, $\{ p_n \}$ is also a Cauchy sequence with respect to the $\| \cdot \|_{\text{RHS}}$-norm.
Invoking the equivalence between the $\| \cdot \|_{\text{RHS}}$- and $\| \cdot \|_{H^m (D)}$-norms in $\mathbb{P}_{m-1} (D)$, we deduce that $\{ p_n \}$ is convergent in the $\| \cdot \|_{\text{RHS}}$-norm.
Therefore, $\{ u_n \}$ also converges in the $\| \cdot \|_{\text{RHS}}$-norm, establishing the completeness of $H^m (D)$ with respect to the $\| \cdot \|_{\text{RHS}}$-norm.
\end{proof}

\bibliographystyle{siamplain}
\bibliography{refs_2mth}
\end{document}